\documentclass[11 pt, twoside, a4paper]{amsart}
\usepackage{amsfonts, amsmath, amssymb, amsthm, amsopn, latexsym, graphicx, mathrsfs,  verbatim, enumerate, url, stmaryrd, enumitem, tikz} 
\usepackage{tikz-cd}

\usepackage{hyperref}
\usepackage[all]{hypcap} 
\usepackage{pst-all} 
\usepackage{array}
\usepackage{Style} 
\usepackage{subcaption}
\usepackage{lastpage}
\usepackage{rotating}
\usepackage{mathtools} 

\usepackage{lineno} 

\usepackage{float} 
\newcolumntype{h}{>{\setbox0=\hbox\bgroup}c<{\egroup}@{}}

\usepackage{tikz}

\usepackage[a4paper, left=2cm,right=2cm,top=2cm,bottom=2cm]{geometry}

\usepackage{index}			
\usepackage{pstricks}			
\usepackage[all,knot,arc]{xy}		

\usepackage{color}

\definecolor{DarkGreen}{RGB}{0, 164, 0}

\newcommand{\dd}{\mathop{}\!\mathrm{d}}

\newcommand{\topdeg}{\operatorname{topdeg}}

\def\nfrac#1#2{%
    \raise.5ex\hbox{\small{$#1$}}%
    \kern-.18em\big/\kern-.2em%
    \lower.5ex\hbox{\small{$#2$}}}

\newcommand{\dgvert}[1]{\varGamma_{#1}^*}
\newcommand{\dgraph}[1]{\varGamma_{#1}}

\newcommand{\Kd}{{a}}
\newcommand{\Kld}{{A}}

\newcommand{\Div}{\on{Div}}

\newcommand{\cen}{\on{cen}}

\newcommand{\Kpi}{{\pi}} 
\newcommand{\Ksigma}{{\sigma}} 

\newcommand{\selfint}[1]{\textcolor{red!60!black}{\scalebox{0.5}{#1}}}
\newcommand{\namediv}[1]{\scalebox{0.7}{$#1$}}
\newcommand{\namepoint}[1]{\textcolor{black}{\scalebox{0.7}{$#1$}}}

\DeclareMathAlphabet{\mathbbm}{U}{bbm}{m}{n}

\graphicspath{{}}
\newcommand{\immdir}{}

\newcommand{\executeiffilenewer}[3]{%
	\ifnum\pdfs
	
	trcmp{\pdffilemoddate{#1}}%
	{\pdffilemoddate{#2}}>0%
	{\immediate\write18{#3}}\fi%
}

\newcommand{\selfarrow}[1][]{\tiny{\!\!\!\xymatrix{\phantom{|}\ar@`{(6,-6),(6,6)}}\hspace{4mm}}}

\newcommand{\selfdasharrow}{\tiny{\!\!\!\xymatrix{\phantom{|}\ar@{-->}@`{(6,-6),(6,6)}}\hspace{4mm}}}

\hypersetup{
colorlinks=true,
linkcolor=blue,
citecolor=red
}

\newtheorem{introthm}{Theorem}

\newtheorem{thm}{Theorem}[section]
\newtheorem{lem}[thm]{Lemma}
\newtheorem{cor}[thm]{Corollary}
\newtheorem{prop}[thm]{Proposition}

\theoremstyle{definition}

\newtheorem{rem}[thm]{Remark}
\newtheorem{defi}[thm]{Definition}
\newtheorem{ex}[thm]{Example}

\theoremstyle{definition}

\newcommand{\Crit}[1]{\mc{C}(#1)} 
\NewDocumentCommand{\mdeg}{+O{} +O{} }
{
{c_{#1}^{#2}}
}

\NewDocumentCommand{\kk}{+m +O{} +O{}}
{
k_{#1 \ifthenelse{\equal{#2}{}}{}{\stackrel{#3}{\to}{#2}}}
}
\NewDocumentCommand{\ee}{+m +O{} +O{}}
{
e_{#1 \ifthenelse{\equal{#2}{}}{}{\stackrel{#3}{\to}{#2}}}
}
\newcommand{\divi}{\mathrm{div}}
\newcommand{\inte}{\mathrm{int}}
\newcommand{\subrelcpct}{\Subset}

\newcommand{\Ediv}[2][]{\mc{E}(#2)_{#1}}

\newcommand{\VC}[1]{\mc{C}_{#1}}

\newcommand{\Frac}{\on{Frac}}					

\newcommand{\Val}[1][]{\mc{V}_{#1}}
\NewDocumentCommand{\NVal}{+O{} +O{} +O{}}
{
\ifthenelse{\equal{#3}{}}{{\mc{V}_{#1}^{#2}}}
{\ifthenelse{\equal{#2}{}}{{\mc{V}_{#1}^{#3}}}{{\mc{V}_{#1}^{#2,#3}}}}
}
\NewDocumentCommand{\CVal}{+O{} +O{} +O{}}
{
\ifthenelse{\equal{#3}{}}{{\hat{\mc{V}}_{#1}^{#2}}}
{\ifthenelse{\equal{#2}{}}{{\hat{\mc{V}}_{#1}^{#3}}}{{\hat{\mc{V}}_{#1}^{#2,#3}}}}
}
\NewDocumentCommand{\FVal}{+O{} +O{}}
{
\ifthenelse{\equal{#2}{}}
{{\hat{\mc{V}}_{#1}^{*}}}
{{\hat{\mc{V}}_{#1}^{*,#2}}}
}

\newcommand{\triv}{\on{triv}}
\newcommand{\var}{{\displaystyle\cdot}}
\NewDocumentCommand{\skel}{+O{} D(){}}
{
\mc{S}_{#1}\ifthenelse{\equal{#2}{}}{}{(#2)}
}
\newcommand{\JD}[1]{R_{#1}}

\newcommand{\mytitle}{Strict germs on normal surface singularities}

\title{\mytitle}

\author{Matteo Ruggiero}
\address{Université de Paris and Sorbonne Université, CNRS, Institut de Mathématiques de Jussieu-Paris Rive Gauche, F-75013 Paris, France}
\email{\href{mailto:matteo.ruggiero@imj-prg.fr}{matteo.ruggiero@imj-prg.fr}}

\date{\today}

\newcommand{\refprop}[1]{\hyperref[prop:#1]{Proposition~\ref*{prop:#1}}}
\newcommand{\refthm}[1]{\hyperref[thm:#1]{Theorem~\ref*{thm:#1}}}
\newcommand{\refcor}[1]{\hyperref[cor:#1]{Corollary~\ref*{cor:#1}}}
\newcommand{\refdef}[1]{\hyperref[def:#1]{Definition~\ref*{def:#1}}}
\newcommand{\refrmk}[1]{\hyperref[rmk:#1]{Remark~\ref*{rmk:#1}}}
\newcommand{\refrem}[1]{\hyperref[rem:#1]{Remark~\ref*{rem:#1}}}
\newcommand{\reflem}[1]{\hyperref[lem:#1]{Lemma~\ref*{lem:#1}}}
\newcommand{\refsec}[1]{\hyperref[sec:#1]{Section~\ref*{sec:#1}}}
\newcommand{\refssec}[1]{\hyperref[ssec:#1]{Subsection~\ref*{ssec:#1}}}

\newcommand{\reffig}[1]{\hyperref[fig:#1]{Figure~\ref*{fig:#1}}}
\newcommand{\refex}[1]{\hyperref[ex:#1]{Example~\ref*{ex:#1}}}
\newcommand{\refsssec}[1]{\hyperref[sssec:#1]{Subsection~\ref*{sssec:#1}}}
\newcommand{\refeqn}[1]{\hyperref[eqn:#1]{Equation~(\ref*{eqn:#1})}}

\begin{document}

\maketitle
\thispagestyle{empty}

\vspace{-5mm}
\begin{center}
\textit{
To Bernard Teissier on the occasion of his 80th birthday.}
\end{center}

\begin{abstract}\noindent
We show that any holomorphic germ $f \colon (X,x_0) \to (Y,y_0)$ of topological degree $1$ between normal surface singularities can be written as $f=\Kpi \circ \Ksigma$, where $\Kpi \colon Y' \to (Y,y_0)$ is a modification and $\Ksigma \colon (X,x_0) \to (Y',y_1)$ is a local isomorphism sending $x_0$ to a point $y_1 \in \Kpi^{-1}(y_0)$.
A result by Fantini, Favre and myself guarantees that when $f$ is a selfmap, then $(X,x_0)$ is a sandwiched singularity. We give here an alternative proof based on the construction of the associated Kato surfaces, and valuative dynamics.

\end{abstract}


\section*{Introduction}

Given a geometric object $X$, one of the most natural tasks is to determine its symmetries.
When $(X,x_0)$ is the germ of a complex analytic space $X$ at a point $x_0$, a natural object to study is the group of analytic automorphisms $\on{Aut}(X,x_0)$.
This space is always huge: a theorem of M\"uller~\cite[p.230--231]{muller:liegroupsanalyticalgebras} states that any singularity $(X,x_0)$ carries countably many analytic vector fields tangent to $(X,x_0)$ and linearly independent.
By taking the flow of such vector fields, we deduce that the automorphism group $\on{Aut}(X,x_0)$ is always infinite-dimensional.

However, the automorphisms built with M\"uller's construction are often small perturbations of the identity (the vector fields have high order at the singularity $x_0$).
The existence of ``special'' automorphisms restricts greatly the geometry of the singularity $(X,x_0)$.
For example, a singularity admits a \emph{contracting} automorphism (i.e., $f(X) \subrelcpct X$ and $\displaystyle \bigcap_{n \geq 0} f^n(X)=\{x_0\}$ for a suitable representant of the germ $f$) if and only if the singularity is \emph{quasihomogeneous} (see \cite{orlik-wagreich:isolatedsingagsurfCaction, camacho-movasati-scardua:quasihomosteinsurfsing,favre-ruggiero:normsurfsingcontrauto} in dimension $2$ and \cite{ornea-verbitsky:algconesLCKmflspot,morvan:singadmitcontrauto} in higher dimensions).


Similar rigidity phenomena appear when one looks at the space $\on{End}(X,x_0)$ of analytic endomorphisms.
For example, while any singularity admits lots of dynamically interesting endomorphisms (see \refrem{hitopnonfinite}), having a finite-to-one selfmap (which is not an automorphism) greatly restricts the geometry of the singularity (it must be \emph{log-canonical}, see \cite{wahl:charnumlinksurfsing, favre:holoselfmapssingratsurf} in dimension $2$, and \cite{boucksom-defernex-favre:volumeisolatedsing, broustet-horing:singendo, zhang:isolatedsingnoninvfiniteendo} in higher dimensions).

In these notes, we focus on the case of \emph{generically injective} selfmaps $f \colon (X,x_0) \selfarrow$, i.e., of topological degree $1$. The ones that are not automorphisms are called \emph{strict germs}.

A special class of strict germs is given by \emph{Kato germs}, i.e., germs of the form $f=\Kpi \circ \Ksigma$, where $\Kpi \colon X' \to (X,x_0)$ is a modification (a proper bimeromorphic map which is an isomorphism outside $x_0$), and $\Ksigma \colon (X,x_0) \to (X',x_1)$ is a local isomorphism, where $x_1 \in \Kpi^{-1}(x_0)$.

When $(X,x_0) \cong (\nC^2,0)$ is regular, it is known by experts that any strict germ $f \colon (\nC^2,0) \to (\nC^2,0)$ admits a Kato decomposition (see, e.g., \cite[Proposition 0.5]{favre:rigidgerms}); however, to my knowledge, an explicit proof of this fact seems to be missing in the literature.
In these notes, we prove this statement in the more general setting of strict germs between normal surface singularities.

\begin{introthm}\label{thm:strictkato}
Let $f\colon (X,x_0) \to (Y,y_0)$ be a strict germ between two normal surface singularities.
Then there exists a modification $\Kpi\colon Y' \to (Y,y_0)$, a point $y_1 \in \Kpi^{-1}(y_0)$, and a local isomorphism $\Ksigma \colon (X,x_0) \to (Y',y_1)$ so that $f=\Kpi \circ \Ksigma$.
\end{introthm}

In particular, any strict selfmap $f \colon (X,x_0) \selfarrow$ of a normal surface singularity is a Kato germ.
Admitting a Kato germ can be seen as a \emph{selfsimilarity} property for a singularity $(X,x_0)$ (other selfsimilarity concepts have been considered in the literature, see, e.g., \cite{defelipe:topspacesval}).
In \cite{fantini-favre-ruggiero:sandwichselfsim}, the authors classify selfsimilar normal surface singularities, showing that they are exactly the \emph{sandwiched} singularities, i.e., singularities that bimeromorphically dominate $(\nC^2,0)$. These singularities play a central role in the proof of the resolution of singularities via Nash transformations (see, e.g., \cite{spivakovsky:sandwichedsing}).

By putting together \refthm{strictkato} and the classification of selfsimilar singularities, we get the following characterization of normal surface singularities admitting strict selfmaps.
 
\begin{introthm}\label{thm:strictsandwiched}
Let $(X,x_0)$ be a normal surface singularity. Then $(X,x_0)$ admits a strict selfmap if and only if $(X,x_0)$ is a sandwiched singularity.
\end{introthm}

\medskip

In these notes, we will present two proofs of \refthm{strictkato}.
The first, essentially due to Charles Favre, 
is based on local flattening techniques of Hironaka, Lejeune-Jalabert and Teissier \cite{hironaka:introrealanal,hironaka:flattening,hironaka-lejeunejalabert-teissier:platificateurlocal}.
The second 
relies on the geometrical properties of the action $f_\bullet \colon \NVal[X] \to \NVal[Y]$ induced on valuation spaces.

For what concerns \refthm{strictsandwiched}, the proof given in \cite{fantini-favre-ruggiero:sandwichselfsim} works over fields of any characteristic, and relies on extending the plumbing techniques of \cite{spivakovsky:sandwichedsing} when working over fields of positive characteristic, on analytic non-archimedean techniques on Berkovich spaces and non-archimedean links, and the combinatorics of self-similar dual graphs.
Here we present an alternative proof (see also \cite[Section 3.4]{ruggiero:HDR}) that works over $\nC$, which is more geometrical in nature: it relies on the geometry of Kato surfaces (suitable compactifications of the orbit spaces of $f$), as well as on the description portrayed in \cite{gignac-ruggiero:locdynnoninvnormsurfsing} of the dynamics of the action induced by selfmaps on valuation spaces.

\medskip

The material is organized as follows. In \refsec{valspaces} and \refsec{valmap} we recall a few basic facts about the valuation space $\NVal[X]$ associated to a normal surface singularity $(X,x_0)$, and the action $f_\bullet \colon \NVal[X] \to \NVal[Y]$ induced by a dominant germ $f \colon (X,x_0) \to (Y,y_0)$. These two sections can easily be skipped by experts.

In \refsec{topdegstrict}, we introduce the topological degree of a dominant germ $f$, and describe the dynamical properties of $f_\bullet$ when $f$ is a strict selfmap. In particular, we recall Gignac-Ruggiero's result of existence and uniqueness of the eigenvaluation $\nu_\star$.

In \refsec{proofstrictkato}, we give the two proofs of \refthm{strictkato}, and comment on some of the difficulties in higher dimensions.
Finally, in \refsec{Katosandwiched}, we give a proof of \refthm{strictsandwiched} assuming that $\nu_\star$ is not divisorial. We prove this latter statement in \refsec{notdivi}.
We conclude with some explicit examples, and comments on the geometry of Kato surfaces depending on the type of their associated eigenvaluation.

\bigskip

\noindent\textbf{Acknowledgements.}

I would like to thank Charles Favre for sharing his notes around the resolution of strict germs in dimension $2$ and explaining the content to me.
It is after discussing with Bernard Teissier about this topic that he made me realize that flattening techniques were not the natural approach in the singular setting, and soon after I had the proof presented here using valuative dynamics.
I thank him for pointing me, as he often does, towards the right direction.
While working on this project, I wandered more deeply in the world of discrepancies: I would like to thank Hussein Mourtada and André Belotto for insightful discussions around this subject, as well as around the flattening theorem.

This work has been partially supported by the ANR grant ``Sintrop'' ANR-22-CE40-0014; the final drafting of this manuscript happened during my recent visit to Junyi Xie at the BICMR:
I am grateful to the institution and to my host for their hospitality.

\section{Basics on valuation spaces}\label{sec:valspaces}

\subsection{Valuation spaces}

Let $(X,x_0)$ be a normal surface singularity. In this section we will recall a few facts about the associated space $\Val[X]$ of (semi-)valuations.

We denote by $R_X$ the completed local ring $\hat{\mc{O}}_{X,x_0}$ and by $\mf{m}_X$ its unique maximal ideal.

\begin{defi}
A \emph{rank one semivaluation} on $R_X$ is a function $\nu\colon R_X\to \R\cup\{+\infty\}$ satisfying $\nu(\phi\psi) = \nu(\phi) + \nu(\psi)$ and $\nu(\phi + \psi)\geq \min\{\nu(\phi), \nu(\psi)\}$ for all $\phi, \psi\in R_X$.

Such a semivaluation is said to be \emph{centered} if in addition $\nu(\phi)\geq 0$ for all $\phi\in R_X$ and $\nu(\phi)>0$ if and only if $\phi\in \mf{m}_X$. The collection of all rank one centered semivaluations on $R_X$ will be denoted by $\CVal[X]$.
\end{defi}

Equivalently, one may define rank one centered semivaluations on $R_X$ not as functions on $R_X$ itself, but on the set of ideals of $R_X$.
The correspondence between the two points of view is given by setting $\nu(\mf{a})=\min\{\nu(\phi)\ |\ \phi \in \mf{a}\}$ (the minimum exists by noetherianity of $R_X$), or conversely $\nu(\phi)=\nu(\phi R_X)$.
In these terms, centered valuations are determined by the conditions $\nu(R_X)=0$ and $\nu(\mf{m}_X)>0$.

The set $\CVal[X]$ is endowed with the so called \emph{weak topology}, which is the weakest topology for which the evaluations $\nu \mapsto \nu(\phi)$ are continuous for any $\phi \in R_X$.

Notice that for any $\lambda > 0$ and $\nu \in \CVal[X]$, we have that $\lambda\nu \in \CVal[X]$ as well. All these valuations $(\lambda \nu)_{\lambda \in (0,+\infty)}$ are distinct, unless $\nu=\triv_{x_0}$ is the trivial valuation, determined by $\triv_{x_0}(\mf{m}_X)=+\infty$.
These observations lead to the following definitions.

\begin{defi}
A semivaluation $\nu\in \hat{\mc{V}}_X$ is \emph{finite} if $\nu(\mf{m}_X)<+\infty$. We denote by $\FVal[X]=\CVal[X] \setminus \{\triv_{x_0}\}$ the set of all finite semivaluations.
	
We also denote by
\[
\Val[X]:=\nP(\CVal[X]) = \FVal[X]/\sim\text.
\]
the set of finite valuations modulo the equivalence relation $\nu \sim \lambda \nu$ for all $\lambda \in (0,+\infty)$.
\end{defi}

The weak topology on $\CVal[X]$ induces one on $\FVal[X]$ by restriction, and on $\NVal[X]$ by quotient.

In the regular setting $(X,x_0)=(\nC^2,0)$, the space $\NVal=\NVal[\nC^2]$, named the \emph{valuative tree} due to its structure as a $\nR$-tree, has been introduced and intensively studied by Favre and Jonsson (see \cite{favre-jonsson:valtree, favre-jonsson:eigenval, jonsson:berkovich}).
We also refer to \cite{favre:holoselfmapssingratsurf, gignac-ruggiero:locdynnoninvnormsurfsing} for the singular setting.

In most cases, it is often useful to fix a normalization for valuation spaces, i.e., a section of the natural projection $\on{pr}:\FVal[X] \to \Val[X]$.
The most common choice is to take the unique element $\nu$ in an equivalence class satisfying $\nu(\mf{m}_X)=1$.
This allows to naturally identify the space $\Val[X]$ with the space
$
\{\nu \in \FVal[X]\ |\ \nu(\mf{m}_X)=1\}$.

\begin{rem}
Semivaluations can also be interpreted as their multiplicative counterpart, seminorms.
In fact, given a semivaluation $\nu \in \FVal[X]$, we can associate the (non-archimedean) seminorm $\abs{\var}_\nu:=e^{-\nu(\var)}$.
This gives a direct correspondence between valuation spaces and Berkovich spaces (see, e.g.,  \cite{berkovich:book}, \cite[Section 4.5]{favre-jonsson:valtree}, or \cite[Section 1]{dujardin-favre-ruggiero:polyskewprod}).
The four types of seminorms in ($1$-dimensional) Berkovich spaces, points of type I, II, III, and IV, correspond respectively to curve semivaluations, divisorial valuations, irrational valuations and infinitely singular valuations, see the discussion below.

In this language, the set $\Val[X]$ corresponds to the non-archimedean link $\on{NL}(X,x_0)$ associated to the local ring $(R_X,\mf{m}_X)$ over $\nC$ considered with the trivial valuation, see \cite{fantini:normalizedlinks, fantini:normalizedBerkovich, thuillier:homotopy, fantini-favre-ruggiero:sandwichselfsim,gignac-ruggiero:locdynnoninvnormsurfsing}. We also refer to \cite{defelipe:topspacesval} for an interpretation of valuation spaces as largest Hausdorff quotient of local Riemann-Zariski spaces.
\end{rem}

\subsection{Types of valuations}\label{ssec:typesval}

We recall here the geometric description of valuations of a normal surface singularity $(X,x_0)$.

\subsubsection*{Modifications}

First, recall that a \emph{modification} $\pi \colon X_\pi \to (X,x_0)$, is a proper bimeromorphic map that is an isomorphism above $X \setminus \{x_0\}$.
Without further mention, we will always assume that $X_\pi$ is normal.
Among modifications, there are:
\begin{itemize}
\item \emph{resolutions}: $X_\pi$ is regular;
\item \emph{good resolutions}: resolutions for which $\pi^{-1}(x_0)$ has simple normal crossings;
\item \emph{log-resolutions (of $\mf{m}_X$)}: good resolutions for which the ideal sheaf $\mf{m}_X \mc{O}_{X_\pi}$ is locally principal.
\end{itemize}

Given a modification $\pi \colon X_\pi \to (X,x_0)$, we denote by $\dgvert{\pi}$ the set of all $\pi$-exceptional primes, i.e., irreducible components of $\pi^{-1}(x_0)$.
We also denote by
\[
\Ediv[\nK]{\pi}=\left\{\sum_{E \in \dgvert{\pi}} \alpha_E E\ :\ \alpha_E \in \nK\right\}
\] the free module of Weil $\nK$-divisors supported on $\pi^{-1}(x_0)$, where $\nK=\nZ$, $\nQ$ or $\nR$.

\subsubsection{Divisorial valuations}

Given a modification $\pi \colon X_\pi \to (X,x_0)$, the order of vanishing $\ord_E$ along $E$ defines a valuation over $\mc{O}_{X_\pi,\eta}$, where $\eta$ is the generic point of $E$.
Its push-forward $\divi_E := \pi_* \ord_E$, defined by $\phi \mapsto \ord_E(\phi \circ \pi)$, and its positive multiples $\lambda \divi_E$, are called \emph{divisorial valuations}; they correspond to points of type II in the Berkovich language.
The space of divisorial valuations is denoted by $\CVal[X][\mathrm{div}]$, and it is a dense subset of $\CVal[X]$. Similarly, the space of normalized divisorial valuations is denoted by $\NVal[X][\mathrm{div}]$, and the normalized equivalent of $\divi_E$ is denoted by $\nu_E$.

\subsubsection{Quasimonomial valuations}\label{sssec:qmval}
More generally, assume that $\pi \colon X_\pi \to (X,x_0)$ is a good resolution; consider a closed point $p \in \pi^{-1}(x_0)$, that we may assume to be \emph{satellite}: $p \in E \cap F$ for two exceptional primes $E,F \in \dgvert{\pi}$.
Take local coordinates $(z,w)$ at $p$ \emph{adapted} to $\pi^{-1}(x_0)$, i.e., such that $E=\{z=0\}$ and $F=\{w=0\}$.
Given a pair $(r,s)$ of non-negative real numbers (not both zero), we consider the monomial valuation $\mu_{r,s}$ at $p$, defined as
\[
\mu_{r,s}\Big(\sum_{i,j}a_{i,j} z^i w^j \Big) := \min\left\{ri+sj\ |\ a_{i,j} \neq 0\right\}\text{.}
\]
Its push-forward $\nu_{r,s}:=\pi_\star \mu_{r,s}$ is called a \emph{quasimonomial valuation}, and the set of all such valuations is denoted by $\CVal[X][\mathrm{qm}]$.

When $(r,s)$ are positive and rationally dependent, that we may assume to be comprime integers modulo equivalence, then $\nu_{r,s}$ is in fact divisorial, associated to the exceptional prime $E_{\frac{s}{r}}$, obtained as the weighted blow-up of $p$ of weight $(r,s)$.
The extreme cases $(r,s)=(1,0)$ and $(r,s)=(0,1)$ also correspond to divisorial valuations, associated to the exceptional primes $E_0=E$ and $E_\infty = F$ respectively.
When $(r,s)$ are rationally independent, the valuation $\nu_{r,s}$ is called \emph{irrational}, and corresponds to points of type III in the Berkovich language. 

\subsubsection{Curve semivaluations}

Consider a branch $C$ at $x_0$, i.e., a germ of (possibly formal) irreducible curve $(C,x_0) \subset (X,x_0)$.
Then we get a semivaluation by setting
\[
\inte_C(\phi):= C \cdot (\phi=0)
\]
where $\cdot$ stands for the local intersection at $x_0$, i.e., $\dim_\nC \frac{R_X}{I_C + \langle \phi \rangle}$.

Any positive multiple of $\inte_C$ is called a \emph{curve (semi-)valuation}; they correspond to type I points (or to the subclass of rigid points, depending on the precise setting in Berkovich theory, see e.g. \cite[Section 1]{dujardin-favre-ruggiero:polyskewprod}).
The prefix (semi-) alludes to the fact that $\inte_C$ takes the value $+\infty$ not only on $0$, but on the whole ideal $I_C$.
The normalized equivalent of $\inte_C$ is denoted by $\nu_C$.

\subsubsection{Infinitely singular valuations}

The remaining valuations are called \emph{infinitely singular}. In terms of Berkovich geometry, they correspond to points of type IV (or to the non-rigid type I points).
There are several ways to characterize them. One is in terms of certain algebraic invariants (they are the ones with rank and rational rank equal to $1$, and transcendental degree equal to $0$); they are also characterized by having a non-finitely generated value group over $\nZ$.
One can think of them as curve valuations associated to curves of ``infinite multiplicity'' (hence the name).

\subsection{Center of a valuation}

One can also characterize valuations in terms of their associated sequence of infinitely near points, built as follows.
Recall that, given a semivaluation $\nu \in \NVal[X]$ (or more precisely, in the sense of Krull valuations, i.e., valuations with values on an arbitrary totally ordered abelian group), and a modification $\pi \colon X_\pi \to (X,x_0)$, the center $\cen_\pi(\nu)$ is the unique scheme-theoretic point $\xi \in X_{\pi_0}$ such that $\mc{O}_{X_\pi,\xi}$ is dominated by the local ring $(R_\nu,\mf{m}_\nu)$, where $R_\nu =\{\phi \in \Frac(R_X)\ |\ \nu(\phi) \geq 0\}$, and $\mf{m}_\nu=\{\phi \in \Frac(R_X)\ |\ \nu(\phi) > 0\}$ is its unique maximal ideal.

Then, start with any resolution $\pi_0 \colon X_{\pi_0} \to (X,x_0)$, and denote by $p_0=\cen_{\pi_0}(\nu)$ the center of $\nu$ in $X_{\pi_0}$.
If $p_0$ is a closed point, take the blow-up of $X_{\pi_1} \to (X_{\pi_0},p_0)$, and set $p_1 = \cen_{\pi_1}(\nu)$.
In this way, we construct a (possibly finite) sequence of models $\pi_n \colon X_n \to (X,x_0)$ and of infinitely-near points $(p_n)_n$ with $p_n \in \pi_n^{-1}(x_0)$.

The type of a valuation $\nu \in \NVal[X]$ is determined by this sequence $(p_n)_n$, as follows:
\begin{itemize}
\item $\nu$ is divisorial iff the sequence $(p_n)$ is finite;
\item $\nu$ is irrational iff $p_n$ is satellite for any $n \gg 0$;
\item $\nu$ is a curve semivaluation iff $p_n$ is free for any $n \gg 0$;
\item $\nu$ is infinitely singular iff the sequence $(p_n)$ contains infinitely many free points and infinitely many satellite points.
\end{itemize}
Notice that replacing the initial resolution $\pi_0$ with any other changes the sequence of infinitely near points, but preserves the properties listed above.

\begin{rem}\label{rem:exccurve}
To be precise, there are two possible cases when $(p_n)_n$ is a sequence of eventually satellite points.
\begin{itemize}
\item Suppose there exists an exceptional prime $E \in \dgvert{\pi_N}$ for some $N \gg 0$ such that $p_n$ belongs to the strict transform of $E$ for any $n > N$.
In this case $\nu$ is an \emph{exceptional curve valuation}, which is a Krull valuation of rank $2$, not associated to a valuation $\nu \in \NVal[X]$, but rather to a tangent vector at the divisorial valuation $\nu_E$.
\item If there is no such exceptional prime $E$ as above, then $\nu$ is an irrational valuation, of the form $(\pi_N)_* \mu$, with $\mu$ a monomial valuation, and $N$ such that $p_n$ is satellite for any $n > N$.
\end{itemize}
\end{rem}

\begin{rem}
Centers give also a geometrical interpretation of the weak topology on $\NVal[X]$.
Given a modification $\pi \colon X_\pi \to (X,x_0)$, and a point $p \in \pi^{-1}(x_0)$, we denote by $U_\pi(p)$ the set of valuations whose center in $X_\pi$ is $p$.
Then this set is weakly open, and the collection of $U_\pi(p)$, when $\pi$ varies among all possible modifications and $p$ among all closed points in $\pi^{-1}(x_0)$, forms a prebase for the weak topology.
\end{rem}

\subsection{Valuation spaces as projective limits of dual graphs}

Given a good resolution $\pi \colon X_\pi \to (X,x_0)$, we can consider the finite set $\skel[\pi]^*=\{\nu_E\ |\ E \in \dgvert{\pi}\}$ of normalized divisorial valuations associated to the exceptional primes of $\pi$.

Their convex hull $\skel[\pi]$ in $\NVal[X]$, called the \emph{skeleton} with respect to $\pi$, is obtained from $\skel[\pi]^*$ by adding all normalized valuations that are monomial in a satellite point of $X_\pi$.
If moreover $\pi$ is a log-resolution of $\mf{m}_X$, then, for any satellite point $p \in E \cap F$, the segment $[\nu_E,\nu_F]_p$ consists of the monomial valuations $\nu_{r,s}=\pi_* \mu_{r,s}$ at $p$ satisfying $rb_E+sb_F=1$, where $b_E=\ord_{E}(\pi^*\mf{m}_X)$ and $b_F=\ord_{F}(\pi^*\mf{m}_X)$.

We deduce that $\skel[\pi]$ has a natural structure of graph, isomorphic to the dual graph $\dgraph{\pi}$ of the resolution $\pi$.
Given $\pi' \colon X_{\pi'} \to (X,x_0)$ another log-resolution of $\mf{m}_X$ dominating $\pi$, we clearly have that $\skel[\pi] \subseteq \skel[\pi']$.
We also have a natural \emph{retraction map} $\rho_{\pi',\pi} \colon \skel[\pi'] \to \skel[\pi]$, defined as follows.
Set $p=\cen_\pi(\nu)$. If $p$ is a free point, belonging to the exceptional prime $E$, then $\rho_{\pi',\pi}(\nu)=\nu_E$.
If $p \in E \cap F$ is a satellite point, then $\rho_{\pi',\pi}(\nu)=\pi_* \mu_{r,s}$ is the (opportunely renormalized) monomial valuation at $p$ with weights $r=\nu(z)$ and $s=\nu(z)$, where $(z,w)$ are local coordinates at $p$ adapted to $\pi^{-1}(x_0)$.

Since any quasimonomial valuation can be found in a skeleton $\skel[\pi]$ for a suitable good resolution $\pi$, we deduce that $\NVal[\mathrm{qm}]$ can be described as the injective limit of all skeleta.
Similarly, the set of normalized valuations $\NVal[X]$ can be described as the projective limt of these skeleta (with respect to the inverse system of good resolutions and the retraction maps).
We refer to \cite{thuillier:homotopy, jonsson:berkovich} for more details.

\section{Map induced on valuation spaces}\label{sec:valmap}

Any dominant holomorphic germ $f \colon (X,x_0) \to (Y,y_0)$ between two normal surface singularities induces a map $f_* \colon \CVal[X] \to \CVal[Y]$ on centered semivaluations, and consequently a map $f_\bullet \colon \NVal[X] \to \NVal[Y]$ at the level of normalized semivaluations.
In this section, we describe some basic properties of these maps, and refer to \cite{favre-jonsson:eigenval, favre:holoselfmapssingratsurf, gignac-ruggiero:locdynnoninvnormsurfsing} for further details.

\subsection{Definitions and first properties}

Let $f \colon (X,x_0) \to (Y,y_0)$ be a holomorphic germ between two normal surface singularities.

This induces a local ring homomorphism $f^*\colon (R_Y,\mf{m}_Y) \to (R_X, \mf{m}_X)$ by pullback: $f^* \psi := \psi \circ f$.
By duality, we get a map $f_*\colon \CVal[X] \to \CVal[Y]$ acting on centered valuations, defined by $f_*\nu := \nu\circ f^*$. This map $f_*$ is continuous with respect to the weak topologies on $\CVal[X]$ and $\CVal[Y]$ and commutes with scalar multiplication of semivaluations.

If $\nu\in \FVal[X]$ is a finite semivaluation then $f_*\nu$ will again be finite, except in exactly one situation: if $\nu$ is a curve semivaluation associated to a curve germ $(C, x_0)\subset (X, x_0)$ which is contracted by $f$, i.e., $f(C)=y_0$.

We denote by $\VC{f}$ the (finite) set of curve semivaluations $\nu_C$ such that $f(C)=y_0$. The case when $\VC{f}=\emptyset$ is rather special: in this case $f$ is called \emph{finite}.

From the discussion above, it follows that $f_*$ induces a (continuous) map $f_\bullet \colon \NVal[X] \setminus \VC{f} \to \NVal[Y]$, which in fact extends by continuity to a map $f_\bullet \colon \NVal[X] \to \NVal[Y]$.

\begin{prop}\label{prop:propertiesfbullet}
Let $f\colon (X,x_0)\to (Y,y_0)$ be any dominant holomorphic germ between normal surface singularities.
\begin{itemize}
\item
If $\nu\in \FVal[X] \setminus \VC{f}$ is not a contracted curve semivaluation, then $f_\bullet\nu$ is of the same type (curve, divisorial, irrational, infinitely singular) as $\nu$.

\item If $\nu$ is a contracted curve semivaluation, then $f_\bullet\nu$ is divisorial.
\end{itemize}

\end{prop}

\subsection{Geometric interpretation of the action}\label{ssec:geomactionval}

For our purposes, we need more geometrical insight about the action of $f_* \colon \NVal[X] \to \NVal[Y]$ on divisorial valuations and curve semivaluations.

We start with divisorial valuations.

\begin{prop}[{\cite[Prop. 2.5]{favre-jonsson:eigenval}}]\label{prop:imagedivi}
Let $f\colon (X,x_0) \to (Y,y_0)$ be a dominant germ between two normal surface singularities, and let $\nu_E \in \NVal[X]$ be a divisorial valuation.
Denote by $\nu_{E'}:= f_\bullet \nu_E$ its image, which is also divisorial by \refprop{propertiesfbullet}. 
Let $\pi \colon X_{\pi}\to (X,x_0)$ and $\varpi \colon Y_{\varpi}\to (Y,y_0)$ be two good resolutions such that $E \in \dgvert{\pi}$ and $E' \in \dgvert{\varpi}$.

Then the lift $\hat{f} \colon X_\pi \dashrightarrow Y_\varpi$ satisfies $\hat{f}|_{E} (E)=E'$, and
\[
f_* \divi_E = \kk{E} \divi E'\text{,}
\] 
where $\kk{E}$ is the ramification index of $\hat{f}|_E$, i.e., the coefficient of $E$ in the divisor $\hat{f}^*E'$.
\end{prop}

If $C$ is a branch at $x_0$, then its image $f(C)$ is either a branch at $y_0$, or is contracted to $y_0$ by $f$.
The next two propositions describe what happens in each case.

\begin{prop}[{\cite[Prop. 2.6]{favre-jonsson:eigenval}}]\label{prop:imageinte}
Let $f \colon (X,x_0) \to (Y,y_0)$ be a dominant germ between normal surface singularities, and let $C$ be a branch at $(X,x_0)$ that is not contracted by $f$.
Then $f(C)=:C'$ is a branch at $(Y,y_0)$, and 
\[
f_* \inte_C = \ee{C} \inte_{C'}\text{,}
\]
where $\ee{C}$ is the local topological degree of $f|_C \colon (C,x_0) \to (C',y_0)$.
\end{prop}

\begin{prop}[{\cite[Prop. 2.7]{favre-jonsson:eigenval}}]\label{prop:imageccurve}
Let $f \colon (X,x_0) \to (Y,y_0)$ be a dominant germ between normal surface singularities, and let $C$ be a branch at $(X,x_0)$ such that $f(C)=y_0$.
By \refprop{propertiesfbullet}, its image $f_\bullet \nu_C = \nu_G$ is divisorial.
Then, for any good resolution $\varpi \colon Y_\varpi \to (Y,y_0)$ such that $G \in \dgvert{\varpi}$, the lift $\hat{f} \colon (X,x_0) \to Y_\pi$ is such that $\hat{f}|_C \colon (C,x_0) \to (G,q)$ for a suitable point $q \in G$.
\end{prop}

\begin{rem}
In terms of Krull valuations, i.e., if we look at the action induced on the Riemann-Zariski spaces, we have that the image of a contracted curve valuation $\nu_C$ is an exceptional curve valuation (see \refrem{exccurve}).
\end{rem}

\subsection{Detecting indeterminacy points}

The map $f_\bullet \colon \NVal[X] \to \NVal[Y]$ also allows to detect the action and indeterminacy points of the lift $\hat{f} \colon X_\pi \dashrightarrow Y_\varpi$ of a dominant germ $f \colon (X,x_0) \to (Y,y_0)$ for any modifications $\pi\colon X_\pi\to (X,x_0)$ and $\varpi\colon Y_\varpi\to (Y,y_0)$.

Recall that for any closed point $p\in \pi^{-1}(x_0)$, then $U_\pi(p)\subset \NVal[X]$ stands for the weak open subset of semivaluations whose center in $X_\pi$ is $p$. 
If $\hat{f}$ is holomorphic at a closed point $p\in \pi^{-1}(x_0)$, and if $q = \hat{f}(p)$, then $f_\bullet(U_\pi(p))\subseteq U_\varpi(q)$; this is simply a matter of unraveling definitions. As a consequence of the valuative criterion of properness (we use here the normality of the spaces $X_\pi$ and $Y_\varpi$), the converse is also true (see, e.g., \cite[Prop.\ 3.2]{favre-jonsson:eigenval} or \cite[Proposition 4.12]{gignac-ruggiero:locdynnoninvnormsurfsing}).

\begin{prop}\label{prop:valcritproperness}
Let $f\colon (X,x_0) \to (Y,y_0)$ be a holomorphic germ between normal surface singularities.
Let $\pi:X_\pi \to (X,x_0)$ and $\varpi:Y_\varpi \to (Y,y_0)$ be two modifications (with $X_\pi$ and $Y_\varpi$ normal).
Then the map $\hat{f}=\varpi^{-1} \circ f \circ \pi$ is holomorphic at a point $p \in \pi^{-1}(x_0)$ if and only if there exists a point $q \in \varpi^{-1}(y_0)$ so that $f_\bullet U_\pi(p) \subseteq U_\varpi(q)$.
In this case, $\hat{f}(p)=q$.
\end{prop}

\subsection{Log-discrepancy}

The log-discrepancy $\Kld_X(\nu)$ of a valuation $\nu \in \FVal[X]$ is a quantity that measures the positivity properties of the canonical bundle $K_X$ along $\nu$, and it is a fundamental tool in birational geometry to measure how bad a singularity is (see, e.g., \cite{kollar-mori:biratgeomalgvar,kollar:singMMP}). It is defined as follows.

Fix a nontrivial holomorphic $2$-form $\omega$ on $(X,x_0)$; the vanishing of $\omega$ defines a Weil divisor on $X$ that we denote by $\Div(\omega)$.
Given a resolution $\pi\colon X_\pi\to (X, x_0)$, the \emph{relative canonical divisor} of $\pi$ is the divisor $K_\pi\in \Ediv[\nQ]{\pi}$ defined by the equality $\Div(\pi^*\omega) = \pi^*\Div(\omega) + K_\pi$, where $\pi^*\Div(\omega)$ refers to the Mumford pull-back of $\Div(\omega)$ (see,  e.g., \cite[p.\ 195]{matsuki:mori-program}). The relative canonical divisor $K_\pi$ does not depend on the choice of $\omega$.

The \emph{log-discrepancy} of an exceptional prime $E \in \dgvert{\pi}$ is given by $\Kld_X(E):=1+\ord_E K_\pi$.
\begin{rem}
Besides a direct computation of log-discrepancies, one can compute them via adjunction formula, as follows.

Fix a good resolution $\pi \colon X_\pi \to (X,x_0)$, and label its exceptional primes $E_1, \ldots, E_n$.
We set $\Kld_i:=\Kld(E_i)$, and denote by $\Kld=(\Kld_i)_i$ the vector of log-discrepancies of the exceptional primes in $\dgvert{\pi}$.
The adjunction formula for $X_\pi$ gives
\begin{equation}\label{eqn:adjunction_can_div}
2g(E)-2 = K_{E} = (K_\pi + E) \cdot E\text{.}
\end{equation}
for every $E \in \dgvert{\pi}$,
where here $g(E)$ denotes the genus of $E$, and $K_E$ denotes the canonical bundle of $E$.
By applying \refeqn{adjunction_can_div} to $E_j$ for all $j=1,\ldots, n$, we get
\[
2g(E_j)-2 = \sum_{i} \Kld_i E_i \cdot E_j - \sum_{i \neq j} E_i \cdot E_j = (M\Kld)_j - s(E_j)\text{,} 
\]
where $s(E_j)$ is the degree of $E_j$ in the graph $\dgraph{\pi}$, and $M$ is the intersection matrix on $\Ediv{\pi}$.
Hence, we can compute the vector of log-discrepancies by the formula
\begin{equation}\label{eqn:computeldadj}
\Kld=M^{-1}(2g-2+s)\text{,}
\end{equation}
where $2g-2+s$ denotes the vector of entries $2g(E_j)-2+s(E_j)$.
\end{rem}

If $\pi'$ is another good resolution dominating $\pi$, i.e., the map $\eta=\pi^{-1} \circ \pi' \colon X_{\pi'} \to X_\pi$ is regular, then we have $\eta_* K_{\pi'} = K_\pi$.
We deduce that the log-discrepancy $\Kld_X(E)$ does not depend on the choice of the model $\pi$ realizing $E$: it is a function of the associated valuation $\ord_E$. 
Log-discrepancy is then defined by homogeneity to all divisorial valuations $\lambda \ord_{E}$, by setting $\Kld(\lambda \ord_{E}):=\lambda \Kld(\ord_{E})$ for any $\lambda > 0$, any good resolution $\pi$ and any exceptional prime $E \in \dgvert{\pi}$.

A very convenient property of log-discrepancies is that they behave well on monomial valuations, in the following sense.
Assume we are in the situation of \refsssec{qmval}, then for any rationally-dependent weights $(r,s)$, we have that $\nu_{r,s}$ is divisorial, and a direct computation yields
\begin{equation}\label{eqn:logdiscrepancywblowup}
\Kld(\nu_{r,s}) = r \Kld(E_0) + s \Kld(E_1)\text{.}
\end{equation}
One can then use \refeqn{logdiscrepancywblowup} to extend the log-discrepancy to a continuous functional on the space $\CVal[X][\mathrm{qm}]$ of quasimonomial valuations, and extend it further to a lower semi-continuous functional $\Kld_X \colon \FVal[X] \to (-\infty,+\infty]$ defined on finite semivaluations (see \cite[Proposition 1.6]{favre:holoselfmapssingratsurf}).

\subsection{Jacobian formula}\label{ssec:jacobianformula}

The Jacobian formula allows to compare the log-discrepancies of a valuation and its image by a dominant map.
A proof of this result can be found in \cite[Section 4.5]{gignac-ruggiero:locdynnoninvnormsurfsing}, see also \cite[Proposition 1.9]{favre:holoselfmapssingratsurf} for the finite case.

\begin{prop}[The Jacobian Formula]\label{prop:jacobian_formula}
Let $(X,x_0)$ and $(Y,y_0)$ be two normal surface singularities, and let $f\colon(X,x_0) \to (Y,y_0)$ be a dominant holomorphic map.
Then for any $\nu \in \FVal[X]$, we have
\begin{equation}\label{eqn:jacobian_formula}
\Kld_Y(f_*\nu)= \Kld_X(\nu) + \nu(\JD{f})\text{,}
\end{equation}
where $\JD{f}$, called the \emph{Jacobian divisor} of $f$, is the Weil divisor on $X$ defined by
\begin{equation}\label{eqn:jacobian_divisor}
\JD{f} = \on{Div}(f^*\omega) - f^* \on{Div}(\omega)\text{,}
\end{equation}
with $\omega$ any holomorphic $2$-form on $(Y,y_0)$. 
\end{prop}

When $(X,x_0) \cong (\nC^2,0)$, then one can take $\omega = \dd x \wedge \dd y$, where $(x,y)$ are local coordinates at $0$. Then $\JD{f}$ is the divisor associated to the jacobian determinant of $f$. In particular, $\JD{f}$ is effective and non-trivial in this case, and $\nu(\JD{f})>0$ for any $\nu \in \FVal[X]$.

Similarly, for finite germs $f \colon (X,x_0) \to (Y,y_0)$, the divisor $\JD{f}$ is effective, even though it might be trivial (see \cite[Theorem B]{favre:holoselfmapssingratsurf}).
In fact, at any point $x \in X \setminus \{x_0\}$, the germ $f \colon (X,x) \to (Y,f(x))$ is between two regular surfaces, and again $\JD{f}$ is locally described at $x$ as the jacobian determinant of $f$.

One can also show that $\JD{f}$ is effective whenever $(Y,y_0)$ is canonical (i.e., the discrepancies $\Kd(E):=\Kld(E)-1$ of all exceptional primes are $\geq 0$), see \cite[Proposition 4.33]{gignac-ruggiero:locdynnoninvnormsurfsing}.

In general however, $\JD{f}$ is not effective, and the difference
$
\Kld_Y(f_*\nu)- \Kld_X(\nu) 
$
might be $\leq 0$.
It is rather simple to find examples of this phenomenon, for example by taking $f$ to be the resolution of a non log-canonical singularity (see \cite[Example 4.29]{gignac-ruggiero:locdynnoninvnormsurfsing}).
One can also construct examples of selfmaps $f \colon (X,x_0) \selfarrow$ on any non log-canonical singularity, or on (the finite quotient of) any cusp singularity (see \cite[Remark 4.32]{gignac-ruggiero:locdynnoninvnormsurfsing}).
See also \cite[Chapter 9]{gignac-ruggiero:locdynnoninvnormsurfsing} for some explicit examples with large topological degree, and \refssec{strictnoeff} for examples of strict selfmaps.

\section{Topological degree and strict germs}\label{sec:topdegstrict}

\subsection{Generalities on topological degree}

Let $f \colon (X,x_0) \to (Y,y_0)$ be a dominant holomorphic germ between two (normal surface) singularities.
When $f$ is \emph{finite},
the topological degree of $f$ is defined as the number of preimages (counted with multiplicity) of a generic point. It can be also computed as the dimension of $\nfrac{R_X}{f^*\mf{m}_Y}$  (see \cite[Section 5.4 Proposition 2]{arnold-guseinzade-varchenko:singdiffmaps1} for a proof in the regular setting).

A third interpretation in the regular setting is given by the degree of the map induced by $f$ on the links $\nS^{3}$ (see \cite[Section 5.2 Theorem 1]{arnold-guseinzade-varchenko:singdiffmaps1}).

In the \emph{generically finite} case (i.e., the germ $f$ is dominant, but not necessarily finite), we consider the following definition of topological degree.

\begin{defi}
Let $f \colon (X,x_0) \to (Y,y_0)$ be a dominant holomorphic germ between two (normal surface) singularities.
We define the \emph{topological degree} of $f$ as:
\[
\topdeg f = \inf_{U} \sup_y \#f|_U^{-1}(y)
\]
where $U$ varies among open neighborhoods of $x_0$ in $U$, and $y$ varies among points in $Y$ which have finite preimage by $f|_U$.
\end{defi}
Notice that the topological degree does not depend on the representant of $f$ chosen; it is a positive integer (we will see that the topological degree is always finite).

\begin{ex}
Given a matrix $A=\begin{pmatrix}
a&b\\
c&d
\end{pmatrix}
,$
with $a,b,c,d \in \nN^*$, we consider the map $f_A \colon (\nC^2,0) \to (\nC^2,0)$ given by $f(z,w)=(z^a w^b, z^c w^d)$.
If $\det A \neq 0$, then $f$ is a dominant germ, which contracts the coordinate axes $\{zw=0\}$ to the origin $0$.
Notice also that the points in $\{zw=0\} \setminus \{0\}$ do not belong to the image of $f$, while for any point outside the coordinate axes, we have exactly $\abs{\det A}$ preimages. We deduce that $\topdeg f = \abs{\det A}$. 
These maps belong to class $6$ of \cite{favre:rigidgerms}.
\end{ex}

\begin{ex}
Fix any $a \geq 2$ and $c \geq 1$.
For any $k \in \nN^*$, we consider the maps $f_k \colon (\nC^2,0) \to (\nC^2,0)$, given by $f_k(z,w)=(z^a, z^c w +z^k)$, which belong to the special case of class $4$ of \cite{favre:rigidgerms}.
The map $f_k$ is in normal form when $k \leq c$ or $k=\frac{ac}{a-1}$, while for the remaining values of $k$, the germ $f_k$ is conjugated to $f_\infty \colon (z,w) \mapsto (z^a,z^c w)$.
By \cite[Proposition 1.5]{favre:rigidgerms}, we have
\[
\topdeg f_k = \begin{cases}
\gcd(a,k) & \text{if } k \leq c\text{,}\\
a & \text{if } k > c\text{.}
\end{cases}
\]
\end{ex}

The interpretation of the topological degree in terms of the dimension of $\nfrac{\mc{O}_X}{f^*\mf{m}_Y}$ does not work in the generically finite setting, since this dimension is infinite in this case.
However, a related concept is the degree $\delta(f)$ of the field extension of $\Frac (R_X)$ over $f^* \Frac(R_Y)$.

\begin{ex}
For the previous examples, a direct computation gives
\[
\delta(f_A)=\abs{\det A} = \topdeg f_A\text{,}
\qquad
\delta(f_k)=p \geq \topdeg f_A\text{.}
\]
\end{ex}

Another way to compute the topological degree comes from the action $f_\bullet \colon \NVal[X] \to \NVal[Y]$ on valuation spaces.
The following result is a consequence of \cite[Theorem  4.18]{gignac-ruggiero:locdynnoninvnormsurfsing}.

\begin{prop}
\label{prop:topdegboundpreimages}
Let $f \colon (X,x_0) \to (Y,y_0)$ be a dominant holomorphic germ between two normal surface singularities.
Then
\[
\topdeg{f}= \max_{\nu \in \NVal[Y]} \# f_\bullet^{-1}(\nu)\text{.}
\]
\end{prop}
In fact, the number of preimages by $f_\bullet$ (counted with multiplicity) of a valuation $\nu \in \NVal[Y]$ is exactly $\topdeg f$ in the finite case, while it can be strictly less for some valuations $\nu$ in the generically finite case (see \cite[Remark 4.19]{gignac-ruggiero:locdynnoninvnormsurfsing}).

\refprop{topdegboundpreimages} gives a sharp bound on a number of preimages of a valuation with respect to \cite[Proposition 2.4]{favre-jonsson:eigenval}, where the authors show that the cardinality of $f_\bullet^{-1}(\nu)$ is bounded by $\delta(f)$.
In particular, we deduce that $\topdeg f \leq \delta(f)$ for any $f$, and we have seen that the inequality can be strict.

The discrepancy between $\delta(f)$ and $\topdeg f$ is due to the fact that the latter measures the number of preimages that remain close to $x_0$, while the former counts also the preimages that escape small neighborhoods of $x_0$.

\subsection{Strict germs}

\begin{defi}
A holomorphic germ $f \colon (X,x_0) \to (Y,y_0)$ between normal surface singularities (or more generally between germs complex analytic spaces of the same pure dimension) is called \emph{strict} if it has topological degree $1$ and it is not an isomorphism.
\end{defi}
\begin{rem}
Sometimes in the literature the family of strict germs contains also local isomorphism.
\end{rem}

As a direct consequence of \refprop{topdegboundpreimages}, we obtain the following.

\begin{cor}\label{cor:strictvalinj}
Let $f \colon (X,x_0) \to (Y,y_0)$ be a strict germ between normal surface singularities. The map $f_\bullet \colon \NVal[X] \to \NVal[Y]$ is injective.
\end{cor}

The next result describes the critical set of a strict germ, showing that it consists of contracted curves.

\begin{lem}\label{lem:imagecrit2dstrict}
Let $f \colon (X,x_0) \to (Y,y_0)$ be a strict germ between two normal surface singularities.
Let $C=\Crit{f}$ be the critical set of $f$. Then $f(C)=y_0$.
\end{lem}
\begin{proof}
Arguing by contradiction, suppose there is $x \in C$ so that $f(x) \neq y_0$.
Being $f$ dominant, there exists a neighborhood $V$ of $y_0$ so that $f^{-1}(y)$ is finite for any $y \in V$.
Up to shrinking the domain of the representative of $f$, we may hence assume that $f$ is finite at $x$.
Since $\on{topdeg}(f)=1$ and $X$ is normal at $x$ (in fact, it is regular), we deduce that $f$ is a local diffeomorphism at $x$, against the assumption $x \in C$.
\end{proof}

In \cite{gignac-ruggiero:locdynnoninvnormsurfsing}, the authors give a complete description of the dynamics of the map $f_\bullet \colon \NVal[X] \selfarrow$ induced by any (non invertible) dominant selfmap $f \colon (X,x_0) \selfarrow$ on a normal surface singularity.
When $f$ is a strict, \reflem{imagecrit2dstrict} ensures that $f$ is not finite, and the dynamics can be described as follows.

\begin{thm}[{\cite[Theorem B]{gignac-ruggiero:locdynnoninvnormsurfsing}}]\label{thm:strictvaldyn}
Let $f \colon (X,x_0) \selfarrow$ be a strict selfmap of a normal surface singularity.
Then there exists a unique $\nu_\star \in \NVal[X]$ such that $f_\bullet^n \nu \to \nu_\star$ (in the weak topology) for any $\nu \in \NVal[X][\mathrm{qm}]$.
\end{thm}
The unique $\nu_\star$ given by \refthm{strictvaldyn} will be referred as \emph{the eigenvaluation} of $f$.

\begin{rem}
We will spend some words about the proof of \refthm{strictvaldyn}. In \cite{gignac-ruggiero:locdynnoninvnormsurfsing}, the authors introduce an extended distance $\rho_X$ on $\NVal[X]$, which takes finite values on $\NVal[X][\mathrm{qm}]$, called \emph{angular distance}.
This distance is defined in terms of intersection of b-divisors associated to valuations, or equivalently as
$$
\rho_X(\nu,\nu'):=\log \beta_X(\nu|\nu') \beta_X (\nu'|\nu)\text{, \qquad where } 
\beta_X(\nu|\nu') := \sup_{\mf{a}} \frac{\nu(\mf{a})}{\nu'(\mf{a})}
$$
is the relative Izumi constant of $\nu$ with respect to $\nu'$, and $\mf{a}$ varies among all $\mf{m}_X$-primary ideals of $R_X$.
The angular distance plays the role of the Poincaré distance for hyperbolic Riemann surfaces, in the sense that for any dominant germ $f \colon (X,x_0) \to (Y,y_0)$, the induced action $f_\bullet \colon \NVal[X] \to \NVal[Y]$ is not expanding for the angular distance:
$$
\rho_Y(f_\bullet \nu, f_\bullet \nu') \leq \rho_X(\nu,\nu')\qquad \forall \nu, \nu' \in \NVal[X]\text{.}
$$
If moreover $f$ is a non-finite (in particular, if $f$ is a strict) selfmap, then $f_\bullet$ is a weak contraction, i.e., we have the strict inequality for any $\nu \neq \nu' \in \NVal[X][\mathrm{qm}]$.
This property is enough to ensure the existence and uniqueness of a locally attracting fixed point $\nu_\star$ (the eigenvaluation), while the equicontinuity of $f_\bullet^n$ with respect to the angular distance guarantees that the basin of attraction at $\nu_\star$ contains all quasimonomial valuations.
\end{rem}


\section{Two proofs of \refthm{strictkato}}\label{sec:proofstrictkato}

\subsection{Via local flattening}


The first proof of \refthm{strictkato} we present here is essentially due to Charles Favre.
The main ingredient is Hironaka Lejeune-Jalabert and Teissier's local flattening theorem. We refer to \cite[Flattening theorem and Theorem 4.4]{hironaka:flattening} for proper morphisms in the global setting, to \cite[Theorem 4]{hironaka-lejeunejalabert-teissier:platificateurlocal}
and to \cite[Theorem 4.4]{hironaka:introrealanal} for the non-proper case (in french and in english).
Here we give the statement in the local setting in dimension 2. 
\begin{thm}[{\textbf{Local flattening Theorem}}]\label{thm:flattening}
Let $f \colon (X,x_0) \to (Y,y_0)$ be a dominant germ between two normal surface singularities. Then there exists a blow-up $\pi_Y$ of a coherent sheaf of ideals $\mc{I}$ supported at $y_0$, and a commutative diagram
\begin{equation}\label{eqn:flatteningdiagram}
\xymatrix@C=2pc@R=2pc@M=3pt@L=3pt{
\hat{X} \ar@{->}[r]^{\hat{f}} \ar@{->}[d]_{\pi_X} &  \hat{Y} \ar@{->}[d]^{\pi_Y} \\
X \ar@{->}[r]^{f} &  Y \\
}
\end{equation}
so that
\begin{itemize}
\item $\pi_X$ is the blow-up of $f^* \mc{I}$,
\item the lift $\hat{f}$ is flat on the fiber $\pi_X^{-1} \circ f^{-1}(y_0)$.
\end{itemize}
\end{thm}
Note that in our setting $f$ is not proper, and in general one could need more than one (but still finitely many) modification $\pi_Y$ to achieve flattening.
But since $Y$ has dimension $2$, the local flattening ideal is locally principal, and in this case one modification suffices.

\begin{proof}[Proof of \refthm{strictkato} (via local flattening)]
Let $f\colon (X,x_0) \to (Y,y_0)$ be a strict germ between two normal surface singularities.
By the Local flattening \refthm{flattening}, we have the commutative diagram of \refeqn{flatteningdiagram}. 
Denote by $C=\mc{C}(f)$ the critical set of $f$. By \reflem{imagecrit2dstrict}, $f(C)=y_0$.
Up to shrinking $X$, we have that $C=f^{-1}(y_0)$, and $f$ induces a biholomorphism between $X \setminus C$ and its image in $Y\setminus \{y_0\}$.
Being $\pi_X$ and $\pi_Y$ modifications over $x_0 \in C$ and $y_0$ respectively, we infer that $\hat{f}$ is a biholomorphism with its image outside of $\pi_X^{-1} (C)$.

We claim that $\hat{f} \colon \hat{X} \to \hat{f}(\hat{X})$ is an isomorphism with its image, that we may suppose relatively compact in $\hat{Y}$ up to shrinking $X$ if necessary. 

Assuming that the claim is proved, we now construct a new surface $Y'$ by gluing together $\hat{Y} \setminus \hat{f}(\hat{X})$ and $X$ through the identification $\hat{f} \circ \pi_X^{-1} \colon \partial X \to \hat{f}(\partial \hat{X})$.

\begin{rem}
More precisely, the gluing procedure is done along small tubular neighborhoods of $\partial X$ and $\hat{f}(\partial \hat{X})$, where $\hat{f} \circ \pi_X^{-1}$ acts as an isomorphism. The isomorphism class of the new surface $Y'$ does not depend on the choice of such neighborhoods.
From now on, we will simply write
\[
Y'= \nfrac{\hat{Y} \setminus \hat{f}(\hat{X}) \sqcup X}{\hat{f} \circ \pi_X^{-1} \colon \partial X \to \hat{f}(\partial\hat{X})}\text{,}
\]
without further mention of the tubular neighborhoods.
\end{rem}

We can also consider the following natural maps.

\begin{itemize}
\item The biholomorphism $\sigma \colon X \to Y'$ obtained by identifying $X$ with its copy inside $Y'$.
\item The map $\pi\colon Y' \to Y$ defined as $\pi_Y$ on $\hat{Y}\setminus \hat{f}(\hat{X})$ and as $f$ on the copy of $X$ inside $Y'$. This is well defined since $\pi_Y \circ \hat{f} \circ \pi_X^{-1} = f$ on $\partial X$.
\item The map $\eta\colon \hat{Y} \to Y'$, defined as the identity on $\hat{Y}\setminus \hat{f}(\hat{X})$ and as $\pi_X \circ \hat{f}^{-1}$ on $\hat{f}(\hat{X})$.
\end{itemize}
They fit in the following commutative diagram:
\[
\xymatrix@C=2pc@R=2pc@M=3pt@L=3pt{
\hat{X} \ar@{->}[rr]^{\hat{f}} \ar@{->}[dd]_{\pi_X} & &  \hat{Y} \ar@{->}[dd]^{\pi_Y} \ar@{->}[dl]^{\eta}\\
& Y' \ar@{->}[dr]^{\pi}& \\
X \ar@{->}[rr]^{f} \ar@{->}[ur]^{\sigma} & & Y \\
}
\]
By construction, $Y'$ is a surface with a singularity at $y_1=\sigma(x_0)$ that is isomorphic to $(X,x_0)$. In particular, up to shrinking $Y$ if necessary, we may assume that $Y'$ is normal, and $y_1$ is the only singularity of $Y'$.
Finally, $\pi$ is a modification over $y_0$, since it is defined as $f$ at a neighborhood of $y_1$. 

It remains to prove the claim. For properties of flat maps on complex spaces, we refer to \cite[Chapter 3]{fischer:cplxanalgeom} or to \cite[Section II.2]{grauert-peternell-remmert:sevcplxvar7}.

Since $\hat{f}$ is flat, it is an open map
;
since $\hat{f}$ is generically finite, we deduce that for any $\hat{y}$ in a small open neighborhood of $E_Y := \pi_Y^{-1}(y_0)$, the fiber $\hat{f}^{-1}(\hat{y})$ is discrete
.
Being $E_Y$ and $E_X:= \pi_X^{-1}(x_0)$ compact sets, we deduce that, up to shrinking $X$ if necessary, the map $\hat{f} \colon \hat{X} \to \hat{Y}$ is finite.
For finite flat maps, the number of preimages of a point $\hat{y}$, counted with multiplicity, is locally constant
. Being $\hat{f}$ generically injective, we deduce that $\hat{f}$ is an injective open map.
But then, at any point $\hat{x} \in E_X$, the map $\hat{f}$ is a injective and flat, hence a local isomorphism
.
We deduce that $\hat{f}$ is an isomorphism with its image.
\end{proof}

\subsection{Via valuation theory}

The second proof of \refthm{strictkato} we present here relies on the geometric description of the action $f_\bullet \colon \NVal[X] \to \NVal[Y]$ described in \refssec{geomactionval}. The proof is rather direct: the idea is to resolve the dynamics along the contracted curve valuations, and take advantage of the rather simple geometry of $2$-dimensional bimeromorphic maps.

\begin{proof}[Proof of \refthm{strictkato} (via valuation theory)]
Let $f\colon (X,x_0) \to (Y,y_0)$ be a strict germ between two normal surface singularities.

By \reflem{imagecrit2dstrict}, the critical set $\Crit{f}$ of $f$ satisfies $f(\Crit{f})=y_0$.
In other terms, any critical branch $C$ is contracted to $y_0$, and its associated curve semivaluation $\nu_C$ is sent by $f_\bullet$ to a divisorial valuation.

Denote by $\VC{f}$ the set of contracted curve semivaluations.
Since $f_\bullet (\VC{f})$ consists of divisorial valuations, there exists a good resolution $\hat{\pi} \colon Y_{\hat{\pi}} \to (Y,y_0)$ such that $f_\bullet (\VC{f}) \subseteq \skel[\hat{\pi}]^*$.

We can then consider the contraction $\eta \colon Y_{\hat{\pi}} \to Y_\pi$ of all exceptional primes $E$ whose associated divisorial valuation $\nu_E$ lies in $\skel[\hat{\pi}]^* \setminus f_\bullet (\VC{f})$.
By construction we have that $Y_\pi$ is normal, and $\hat{\pi}= \pi \circ \eta$ for a suitable modification $\pi \colon Y_\pi \to (Y,y_0)$ such that $f_\bullet (\VC{f}) = \skel[\pi]^*$.

We claim that $f \colon (X,x_0) \to (Y,y_0)$ lifts to a local isomorphism $\sigma \colon (X,x_0) \to (Y_\pi,y_1)$, for a suitable $y_1 \in \pi^{-1}(y_0)$, as resumed by the following commutative diagram:
\[
\xymatrix@C=6pt@R=20pt@M=3pt@L=3pt{
 & &  Y_{\hat{\pi}} \ar@{->}[dd]^{\hat{\pi}} \ar@{->}[dl]^{\eta}\\
& Y_\pi \ar@{->}[dr]^{\pi}& \\
(X,x_0) \ar@{->}[rr]^{f} \ar@{->}[ur]^{\sigma} & & (Y,y_0) \\
}
\]

In fact, being $f_\bullet$ injective by \refcor{strictvalinj}, we have that the image of $\NVal[X] \setminus \VC{f}$ by $f_\bullet$ is contained in $\NVal[Y] \setminus f_\bullet (\VC{f})$. Being $f_\bullet$ continuous and $\NVal[X] \setminus \VC{f}$ connected, we deduce that $f_\bullet(\NVal[X] \setminus \VC{f})$ is contained in a connected component of $\NVal[Y] \setminus f_\bullet (\VC{f}) = \NVal[Y] \setminus \skel[\pi]^*$, which is of the form $U_\pi(y_1)$ for a suitable $y_1 \in \pi^{-1}(y_0) \subset Y'$.

By continuity of $f_\bullet$ we deduce that $f_\bullet(\NVal[X]) \subseteq \overline{U_\pi(y_1)}$, and by the valuative criterion of properness \refprop{valcritproperness} we infer that the bimeromorphic map $\sigma := \pi^{-1} \circ f \colon (X,x_0) \dashrightarrow Y'$ is in fact regular, and $\sigma(x_0)=y_1$.

Moreover, by construction, $\sigma$ does not contract any curve, and we deduce that $\sigma$ is a local isomorphism. 
\end{proof}

\subsection{Remarks}

Let $f \colon (X,x_0) \to (Y,y_0)$ be a dominant map between normal analytic singularities of the same dimension $d$ and topological degree $e$.
We have showed that if $d=2$ and $e=1$, then $f$ factorizes as $f=\Kpi \circ \Ksigma$, with $\Kpi$ a modification and $\Ksigma$ a local isomorphism.
There are several ways one could try to generalize this result.

First, notice that local isomorphisms are exactly the germs of topological degree $1$ that are also finite (we are working with normal singularities).
Then one could wonder if it is true that any map of topological degree $e$ can be factorized as $f=\Kpi \circ \Ksigma$, with $\Kpi \colon Y' \to (Y,y_0)$ a modification, and $\Ksigma \colon (X,x_0) \to (Y',y_1)$ finite.

If such a decomposition exists, then in particular we must have that $f_\bullet (\NVal[X])$ is contained in a proper closed subset $\overline{U} \varsubsetneq \NVal[Y]$ with $U=U_\Kpi(x_1)$.
But there exists non-finite maps for which $f_\bullet$ is surjective.

\begin{ex}
The map $f \colon (\nC^2,0) \selfarrow$ defined by $f(x,y)=\big(x(y+x^3),x^2y\big)$ has topological degree $5$, and is not finite, since $f(\{x=0\})=0$; its induced action $f_\bullet \colon \NVal \selfarrow$ on the valuative tree is surjective. See \cite[Section 9.2]{gignac-ruggiero:locdynnoninvnormsurfsing} for detailed computations.
\end{ex}

\begin{rem}\label{rem:hitopnonfinite}
Any singularity $(X,x_0)$ admits non-finite selfmaps of high topological degree of the form $f = \Kpi \circ \Ksigma$ with $\Kpi$ modification and $\Ksigma$ finite.
To see that, take a resolution $\Kpi \colon X' \to (X,x_0)$; then consider $\ell \colon (X,x_0) \to (\nC^d,0)$ to be an embedding $(X,x_0) \hookrightarrow (\nC^D,0)$ followed by a generic projection $(\nC^D,0) \to (\nC^d,0)$. The genericity ensures that $\ell$ is finite.
Finally, pick any dominant, finite morphism $g \colon (\nC^d,0) \to (X',x_1)$, where $x_1 \in \pi^{-1}(x_0)$.
The composition $f=\Kpi \circ \underbrace{g \circ \ell}_{\Ksigma} \colon (X,x_0) \selfarrow$ gives a map of the desired form.
\end{rem}

\begin{rem}
Focussing on the case $e=1$, the two proofs of \refthm{strictkato} don't work as well in higher dimensions: since $f$ is not finite, the local flattening theorem requires several local blow-ups, possibly of non-compact centers, to get flatness.
As for the approach with valuation spaces, it makes strong use of the local tree structure of $\NVal[X]$; in higher dimension however $\NVal[X]$ is a projective limit of higher dimensional complexes.
Moroever, one needs also to be careful to new phenomena due to the absence of strong factorizations of birational maps.
\end{rem}
 
\begin{ex}
The map $f \colon (\nC^3,0) \selfarrow$ defined by
\[
f(x,y,z)=\big(x^3y^2z^3, x^2yz^2,x^2y^2z\big)
\]
has topological degree $1$, and contracts all the coordinate hyperplanes to $\{0\}$. It defines in local coordinates a toric modification of $(\nC^3,0)$, but this modification does not factorize through the blow-up of the origin.
See \cite[Example 3.6]{istrati-otiman-pontecorvo-ruggiero:torickatomanifolds} for further details.
\end{ex}

\section{Proof of \refthm{strictsandwiched}}\label{sec:Katosandwiched}

We proceed here with the proof of \refthm{strictsandwiched}. We start by reporting the construction of Kato data on sandwiched singularities, that can be found in \cite[Section 9.3]{fantini-favre-ruggiero:sandwichselfsim}.
We then prove that a normal surface singularity admitting a Kato germ $f$ is sandwiched, under the assumption that the eigenvaluation $\nu_\star$ of $f$ is not divisorial.
In the next section, we will show that $\nu_\star$ is not divisorial, thus concluding the proof.

\subsection{Sandwiched singularities are selfsimilar}\label{ssec:sandwichedselfsim}

\newcommand{\piY}{\wt{\Kpi}}
\newcommand{\sigmaY}{\wt{\Ksigma}}

Sandwiched singularities are all obtained as follows (see \cite[Section II.1]{spivakovsky:sandwichedsing}).
\begin{itemize}
\item First, we take a sequence of blow-ups $\piY \colon Y' \to (Y,y_0)= (\nC^2,0)$.
\item We select any exceptional divisor $D$, whose support $\abs{D} \subseteq \piY^{-1}(y_0)$ is connected.
By Grauert's criterion \cite{grauert:ubermodifikationen}, we can contract $D$, to obtain a bimeromorphic map $\mu \colon Y' \to X$, where $X$ has a normal singularity at $x_0:=\mu(D)$ (notice that the normality condition ensures the uniqueness of $\mu$ up to isomorphisms).
\item We contract the part of the exceptional divisor $\piY^{-1}(y_0)$ that we haven't contracted before, thus completing the diagram $\piY= \eta \circ \mu$ with $\eta\colon X \to (Y,y_0)$ the proper bimeromorphic map that gives to $(X,x_0)$ the sandwiched structure.
\end{itemize}

Let now $\sigmaY \colon (Y,y_0) \to (Y',y_1)$ be any local isomorphism, where $y_1 \in \abs{D}$ is any point in the support of $D$.
In particular, the pair $(\piY,\sigmaY)$ is a Kato datum over $(Y,y_0)$.

In this setting, we can lift the Kato datum to $(X,x_0)$ via the bimeromorphic map $\eta$, as follows (see also \reffig{sandkato}).

\begin{itemize}
\item We construct the space $X'$, that will be the total space of the Kato datum over $(X,x_0)$, by fibered product with respect to $\pi$ and $\eta$:
$$
X':= \nfrac{Y' \setminus \sigma(Y) \sqcup X}{\sigmaY \circ \eta \colon \partial X \to \sigmaY(\partial Y)}\text{.}
$$
\item We denote by $\Ksigma \colon X \to X'$ the natural identification between $X$ and its copy inside $X'$. Clearly $\Ksigma$ is an isomorphism.
\item We define the map $\Kpi \colon X' \to X$ as:
$$
\Kpi \equiv
\begin{cases}
\mu &\text{on } Y' \setminus \sigmaY(Y)\text{,}\\
\mu \circ \sigmaY \circ \eta &\text{on } \Ksigma(X)\text{.}
\end{cases}
$$
\end{itemize}
The map $\Kpi$ is a non-trivial bimeromorphic map from $X'$ to $X$, and the pair $(\Kpi,\Ksigma)$ is a Kato datum over $(X,x_0)$, as desired.

\begin{center}
\begin{figure}[htb]
\centering
\def\svgwidth{.6\columnwidth}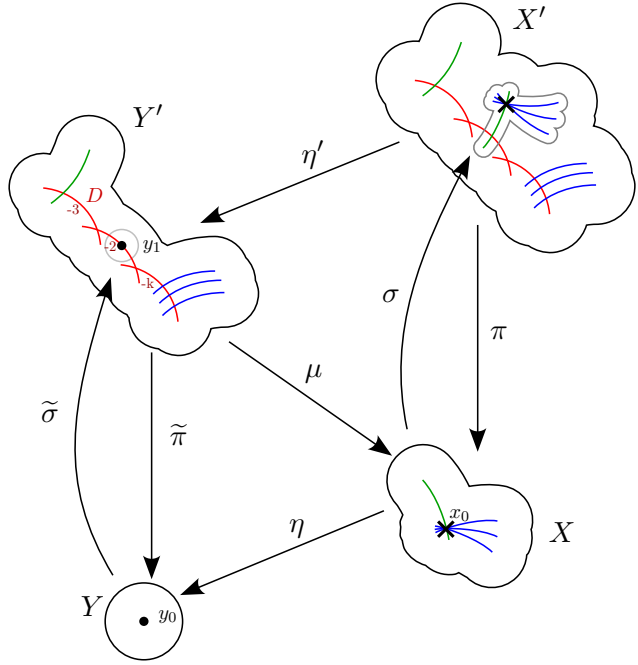
\caption{Example of a Kato datum on the quotient singularity of \refssec{strictnoeff} seen as a sandwiched singularity. Self-intersections on curves are indicated for $Y'$, they are omitted for ($-1$)-curves.}
\label{fig:sandkato}
\end{figure}
\end{center}

\begin{rem}\label{rem:semiconj}
Notice that the germ $\eta \colon (X,x_0) \to (\nC^2,0)$ at $x_0$ gives a semiconjugation between $f=\Kpi \circ \Ksigma \colon (X,x_0) \selfarrow$ and $\wt{f}=\piY \circ \sigmaY \colon (\nC^2,0) \selfarrow$.
\end{rem}

\subsection{A strategy to prove that selfsimilar singularities are sandwiched}

Let now $(X,x_0)$ be a normal surface singularity admitting a strict germ $f \colon (X,x_0) \selfarrow$.
By \refthm{strictkato}, $f$ is a Kato germ: it decomposes as $f=\Kpi \circ \Ksigma$, where $\Kpi\colon X'\to (X,x_0)$ is a non-trivial modification, and $\Ksigma \colon (X,x_0) \to (X',x_1)$ is a local isomorphism.

Let $\mu \colon Z \to (X,x_0)$ be the minimal good resolution of $(X,x_0)$ (or in fact any good resolution).
While $\mu^{-1}\colon X \dashrightarrow Z$ has clearly an indeterminacy point at $x_0$, the lift $\Phi:= \mu^{-1} \circ \Kpi: X' \dashrightarrow Z$ might be regular at $x_1$.
If this is the case, then we are done, since $(X,x_0) \cong (X',x_1)$ would dominate, via $\Phi$, the smooth model $Z$.

While this does not happen in general, it is natural to replace the germ $f$ with an iterate $f^n$, and consider a Kato datum $(\Kpi^{(n)},\Ksigma^{(n)})$ so that $f^n = \Kpi^{(n)} \circ \Ksigma^{(n)}$.

\subsection{Composition of Kato data}

Given two Kato data $(\Kpi_1,\Ksigma_1)$ and $(\Kpi_2, \Ksigma_2)$ associated to Kato germs $f_1$, $f_2$ on the same singularity $(X,x_0)$, one can construct a new Kato datum $(\Kpi,\Ksigma)$ so that $\Kpi \circ \Ksigma = f_1 \circ f_2$, as follows.

\begin{itemize}
\item We construct the space $X''$, that will be the total space of the Kato datum over $(\Kpi,\Ksigma)$, as:
$$
X'':= \nfrac{X'_1 \setminus \Ksigma_1(X) \sqcup X'_2}{\Ksigma_1 \circ \Kpi_2 \colon \partial X'_2 \to \Ksigma_1(\partial X)}\text{.}
$$
\item We denote by $\Ksigma'_2\colon\Ksigma_1(X) \subseteq X'_1 \to X''$ the map $\Ksigma'_2 :=  \Ksigma_2 \circ \Ksigma_1^{-1}$ (up to the natural identification between $X'_2$ and its copy inside $X''$). 
\item We define the map $\Kpi'_2 \colon X'' \to X'_1$ as:
$$
\Kpi'_2 \equiv
\begin{cases}
\id &\text{su } X'_1 \setminus \Ksigma_1(X)\text{,}\\
\Ksigma_1 \circ \Kpi_2 &\text{su } X'_2\text{,}
\end{cases}
$$
which is a bimeromorphic map, that is a local isomorphism outside the exceptional divisor of $\Kpi_2$ in (the copy of $X'_2$).
\end{itemize}
The composition $\Ksigma:= \Ksigma'_2 \circ \Ksigma_1$ is a well defined local isomorphism.
Analogously, the map $\Kpi:=\Kpi_1 \circ \Kpi'_2$ is a bimeromorphic map from $X''$ to $X$.
By construction, we also have $\Kpi(\Ksigma(x_0))=x_0$, and the pair $(\Kpi,\Ksigma)$ is a Kato datum over $(X,x_0)$ (with $\Kpi$ that is a local isomorphism if and only if both $\Kpi_1$ and $\Kpi_2$ are).

The situation is summarised by the diagram:
$$
\xymatrix@C=8mm@R=12mm
{
& X'' \ar@{->}[dl]^{\Kpi'_2} \ar[dd]^\Kpi& \\
X'_1 \ar@{->}[dr]^{\Kpi_1}\ar@/^1pc/@{->}[ur]^{\Ksigma'_2} & & X'_2\ar@{->}[dl]_{\Kpi_2}\\
& (X,x_0)\ar@/^1pc/@{->}[ul]^{\Ksigma_1} \ar@/_1pc/@{->}[ur]_{\Ksigma_2} \ar@/^1pc/@{->}[uu]^{\Ksigma} 
}
$$
A direct computation yields
$$
\Kpi \circ \Ksigma 
= \Kpi_1 \circ \Kpi'_2 \circ \Ksigma'_2 \circ \Ksigma_1
= \Kpi_1 \circ \Ksigma_1 \circ \Kpi_2 \circ \Ksigma_2 \circ \Ksigma_1^{-1} \circ \Ksigma_1
= f_1 \circ f_2\text{.}
$$

\subsection{Proof of \refthm{strictsandwiched} assuming the eigenvaluation is not divisorial}

Consider the tower of Kato data $(\Kpi^{(n)}, \Ksigma^{(n)})$ associated to $f^n$, and set $\Phi_n:= \mu^{-1} \circ \Kpi^{(n)} \colon X^{(n)} \dashrightarrow Z$.
Our goal is to show that $\Phi_n$ is regular at $x_n := \Ksigma^{(n)}(x_0)$ for $n \gg 0$.
The situation can be summarized by the following diagram:
\[
\xymatrix@C=15mm@R=6mm
{
{(X^{(n)},x_n)}\ar@{->}[d]^{\Kpi_n} \ar@{->}@[red][]!<2ex,-2ex>;[ddddr]^{\Phi_n\  \textcolor{red}{?}} & \\
\phantom{*}{\vdots} \phantom{*}\ar@{->}[d]^{\Kpi_3} \ar@/^1pc/@{->}[u]^{\sigma_n}& \\
(X'',x_2) \ar@{->}[d]^{\Kpi_2} \ar@/^1pc/@{->}[u]^{\Ksigma_3} \ar@{-->}[ddr]^{\Phi_2}& \\
(X',x_1) \ar@{->}[dd]^{\Kpi} \ar@/^1pc/@{->}[u]^{\Ksigma_2} \ar@{-->}[dr]^{\Phi} & \\
& Z \ar@{->}[dl]_-{\mu} \\
(X,x_0)\ar@/^1pc/@{->}[uu]^{\Ksigma} 
}
\]
where $\Ksigma^{(n)}:= \Ksigma_n \circ \cdots \circ \Ksigma_2 \circ \Ksigma$ and $\Kpi^{(n)}:= \Kpi \circ \Kpi_2 \circ \cdots \circ \Kpi_n$.

By \refprop{valcritproperness}, we have that $\Phi_n$ is regular at $x_n$ if and only if
$
U_{\Kpi^{(n)}}(x_n)
$
is contained in a connected component of $\NVal[X] \setminus \mc{S}_\mu^*$, or equivalently, if and only if
$$
U_{\Kpi^{(n)}}(x_n) \cap \mc{S}_\mu^\star = \emptyset\text{.}
$$
Notice that by construction we have $f^n_\bullet (\NVal[X]) = \overline{U_{\Kpi^{(n)}}(x_n)}\text{.}$

We now conclude the proof of \refthm{strictsandwiched}, assuming that the eigenvaluation $\nu_\star$ is not divisorial. 

\begin{prop}\label{prop:selfsimilarifeigenvalnotdiv}
In the situation above, suppose that the eigenvaluation $\nu_\star$ of $f$ does not belong to $\mc{S}_\mu^*$. Then there exists $n \gg 0$ such that $f^n_\bullet(\NVal[X]) \cap \mc{S}_\mu^* = \emptyset$.
\end{prop}

\newcommand{\piZ}{\hat{\Kpi}}
\newcommand{\sigmaZ}{\hat{\Ksigma}}

\begin{proof}
Let $(\Kpi,\Ksigma)$ be a Kato datum over $(X,x_0)$, and let $\mu\colon Z \to (X,x_0)$ be a good resolution as in the statement.
Let $U$ be the connected component of $\NVal[X]\setminus \mc{S}_\mu^*$ containing $\nu_\star$. By \refthm{strictvaldyn}, there exists $n \gg 1$ such that $f_\bullet^n(\nu_E) \in U$ for any $E \in \dgvert{\mu}$.
Up to replacing $(f,\Kpi,\Ksigma)$ by $(f^n,\Kpi^{(n)}, \Ksigma^{(n)})$, we may assume $n = 1$.

We resolve the singularities at $x_0$ and $x_1=\Ksigma(x_0)$ simultaneously, by lifting the 
Kato datum $(\Kpi,\Ksigma)$, with respect to a resolution $\mu \colon Z \to (X,x_0)$, as follows.
\begin{itemize}
\item We construct the space $Z'$, by gluing together $X' \setminus \Ksigma(X)$ and a copy of $Z$ via the natural identification $\Ksigma \circ \mu$: 
\[
Z':= \nfrac{X' \setminus \Ksigma(X) \sqcup Z}{\Ksigma \circ \mu \colon \partial Z\to \Ksigma(\partial X)}\text{.}
\]
\item We denote by $\sigmaZ \colon Z \to Z'$ the natural identification between $Z$ and its copy inside $Z'$.
\item We define the map $\piZ \colon Z' \to Z$ as:
\[
\piZ \equiv
\begin{cases}
\id &\text{on } X' \setminus \Ksigma(X)\text{,}\\
\Ksigma \circ \mu &\text{on } Z\text{.}
\end{cases}
\]
\end{itemize}

Notice that $\sigmaZ$ is an isomorphism, while $\piZ$ is a proper bimeromorphic map. In particular, $\mu \circ \piZ$ dominates $\mu$.
In terms of skeleta, this condition is equivalent to asking that $\mc{S}_\mu^* \subseteq \mc{S}_{\mu \circ \piZ}^*$.
Set $S:=f_\bullet \mc{S}_\mu^*$.
But then, $f_\bullet(\NVal[X])$ is contained in the closure of the connected component of $\NVal[X] \setminus \big(\mc{S}_{\mu \circ \piZ}^* \setminus S\big)$ containing $S$.
Since by assumption $S \cap \mc{S}_\mu^* = \emptyset$, we deduce that $U_\Kpi(x_1) \subseteq \NVal[X]\setminus \mc{S}_\mu^*$, and we are done.
\end{proof}

\section{The eigenvaluation $\nu_\star$ is not divisorial}\label{sec:notdivi}

\subsection{The regular case}

The goal of this section is to prove the following result, which concludes the proof of \refthm{strictsandwiched}.

\begin{thm}\label{thm:nustarnotdivisorial}
Let $f\colon(X,x_0) \selfarrow $ be a strict selfmap of a normal surface singularity.
Then the eigenvaluation $\nu_\star$ of $f$ is not divisorial.
\end{thm}

We start by presenting a proof of \refthm{nustarnotdivisorial} that works when $(X,x_0) \cong (\nC^2,0)$ is regular, based on the Jacobian formula, that governs how $f_*$ acts with respect to log-discrepancy.

\begin{proof}[Proof of \refthm{nustarnotdivisorial} in the regular case]
Suppose now by contradiction that the eigenvaluation $\nu_\star$ is divisorial: $\nu_\star = \nu_{E_\star}$ for an exceptional prime $E_\star$.

By \refprop{imagedivi} we have that $f_*\nu_{E_\star} = \kk{E_\star} \nu_{E_\star}$.
Since $f$ is strict, we also have $\kk{E_\star} = 1$, hence $f_*\nu_{E_\star} = \nu_{E_\star}$.
We deduce from the Jacobian formula \refeqn{jacobian_formula} that
\[
\Kld(\nu_\star) = \Kld(\nu_\star) + \nu_\star(\JD{f})\text{.}
\] 
Being $\nu_\star$ divisorial, we have that $\Kld(\nu_\star)$ is finite; being $(X,x_0) \cong (\nC^2,0)$ regular, we have that $\JD{f}$ is effective, and $\nu_\star(\JD{f})>0$: a contradiction.

\end{proof}

Unfortunately, this proof does not carry on to the singular setting, where the Jacobian divisor $\JD{f}$ is not anymore effective (see \refssec{jacobianformula} and \refssec{strictnoeff}).


\subsection{The singular case}

We present here a proof of \refthm{nustarnotdivisorial} when $(X,x_0)$ is an arbitrary normal surface singularity. This proof is an adaptation of the arguments used in \cite{fantini-favre-ruggiero:sandwichselfsim}, where we take advantage of the informations given by \refthm{strictvaldyn}. It requires \refthm{strictkato} in order to reduce to the case of a Kato germ $f=\Kpi \circ \Ksigma$.

\begin{proof}[Proof of \refthm{nustarnotdivisorial} in the singular case.] 
Let $f \colon (X,x_0) \selfarrow$ be a strict germ, and suppose by contradiction that its eigenvaluation is divisorial: $\nu_\star=\nu_E$, for some $E \in \dgvert{\mu}$, with $\mu \colon Z \to (X,x_0)$ a good resolution of $(X,x_0)$.
Up to further blow-ups, we may assume that $E$ is not the only exceptional prime in $\dgvert{\mu}$.

By \refthm{strictkato}, $f$ is a Kato germ, associated to a Kato datum $(\Kpi,\Ksigma)$.
As in the proof of \refprop{selfsimilarifeigenvalnotdiv}, we lift this Kato datum with respect to $\mu$, and get a modification $\piZ \colon Z' \to Z$, which is a composition of point blow-ups (being $Z$ and $Z'$ regular), and an embedding $\sigmaZ \colon Z \rightarrow Z'$ that fit the following commutative diagram:
$$
\xymatrix@C=10mm@R=5mm
{
& Z' \ar@{->}[dl]_-{\mu'} \ar@{->}[dd]^{\piZ} \\
(X',x_1) \ar@{->}[dd]^{\Kpi} & \\
& Z \ar@{->}[dl]_-{\mu} \ar@/^1pc/@{->}[uu]^{\sigmaZ}\\
(X,x_0)\ar@/^1pc/@{->}[uu]^{\Ksigma} 
}
$$
Since $f_\bullet \nu_E = \nu_E$, we must have that $\sigmaZ(E)=E'$ the strict transform of $E$ in $Z'$.
Being $\sigmaZ$ an isomorphism defined on a neighborhood of $E$, we have that $E \cdot E = E' \cdot E'$, and we deduce that $\piZ$ does not blow-up any point in $E$.

We deduce that any exceptional prime $L_1(E):=\{D \in \dgvert{\mu}\setminus \{E\}\ |\ D \cap E \neq \emptyset\}$ is sent by $\sigmaZ$ to the strict transform of another exceptional prime in $L_1(E)$.
In other terms, $\{\nu_D\ |\ D \in L_1(E)\}$ is $f_\bullet$-invariant.
By assumption, $L_1(E)$ contains at least one component $D$, and up to replacing $f$ by an iterate, we get that $f_\bullet \nu_D = \nu_D$.
This contradicts the uniqueness of the eigenvaluation $\nu_\star = \nu_E$.
\end{proof}

\subsection{A strict germ with non-effective Jacobian divisor}\label{ssec:strictnoeff}

For any integer $k \geq 2$, we consider the quotient singularity $(X,x_0)$ whose minimal resolution $\mu \colon Z \to (X,x_0)$ has a chain of three rational curves $E_1, E_2, E_3$, of self-intersection $(-3,-2,-k)$.
From the continued fraction
$$
3-\frac{1}{2-\frac{1}{k}} = \frac{5k-3}{2k-1}\text{,}
$$
we deduce that $(X,x_0)$ is of type $\frac{1}{5k-3} (1,2k-1)$, i.e., setting $p=5k-3$ and $q=2k-1$, we have that $(X,x_0)$ it is isomorphic to $\nfrac{\nC^2}{G}$ with $G=\nfrac{\nZ}{p\nZ}$, acting on $\nC^2$ as $(z,w) \mapsto (\zeta z, \zeta^q w)$, with $\zeta$ a $p$-th root of unity.

If we denote by $M$ the intersection matrix associated ot the minimal resolution, we get
\[
M=\begin{pmatrix}
-3 & 1 & 0\\
 1 &-2 & 1\\
 0 & 1 & -k
\end{pmatrix}
\text{,}
\qquad
M^{-1}=
- \frac{1}{p}
\begin{pmatrix}
2k-1& k  & 1\\
 k 	& 3k & 3\\
 1 	& 3  & 5
\end{pmatrix}
\text{,}
\]
and from \refeqn{computeldadj} we deduce
\[
\Kld(E_1)=\frac{2k}{p},\qquad \Kld(E_2)=\frac{k+3}{p},\qquad \Kld(E_3)=\frac{6}{p}\text{.}
\]
Notice that $\Kld(E_j)<1$ for any $j$, and we deduce that $(X,x_0)$ is not canonical.

Let now $\piZ \colon Z' \to Z$ be the following sequence of blow-ups.
First, we blow-up a free point in $E_2$, obtaining the new exceptional prime $E'_3$. Then blow-up a free point in $E'_3$, obtaining the new exceptional prime $E'_4$. Finally we blow-up $k-1$ free points in $E'_4$, obtaining exceptional primes $D'_1, \ldots, D'_{k-1}$.

The exceptional primes $E_2, E'_3, E'_4$ form a chain of rational curves of self-intersection $(-3,-2,-k)$. Being cyclic quotient singularities taut, there exists an isomorphism $\sigmaZ \colon U \to U'$ from a neighborhood $U$ of $E:=E_1 \cup E_2 \cup E_3$ in $Z$ to a neighborhood of $E':=E_2 \cup E'_3 \cup E'_4$ in $Z'$.
By contracting $E$ we get the original quotient singularity $(X,x_0)$, while contracting $E'$ gives a normal surface $X'$, with a quotient singularity at $x_1$. By construction, $\piZ$ descends to a modification $\Kpi \colon X' \to (X,x_0)$, and $\sigmaZ$ induces a local isomorphism $\Ksigma \colon (X,x_0) \to (X',x_1)$.
Hence we get a Kato germ $f = \Kpi \circ \Ksigma \colon (X,x_0) \selfarrow$ (see \reffig{quot}).

\begin{center}
\begin{figure}[htb]
\centering
\def\svgwidth{.65\columnwidth}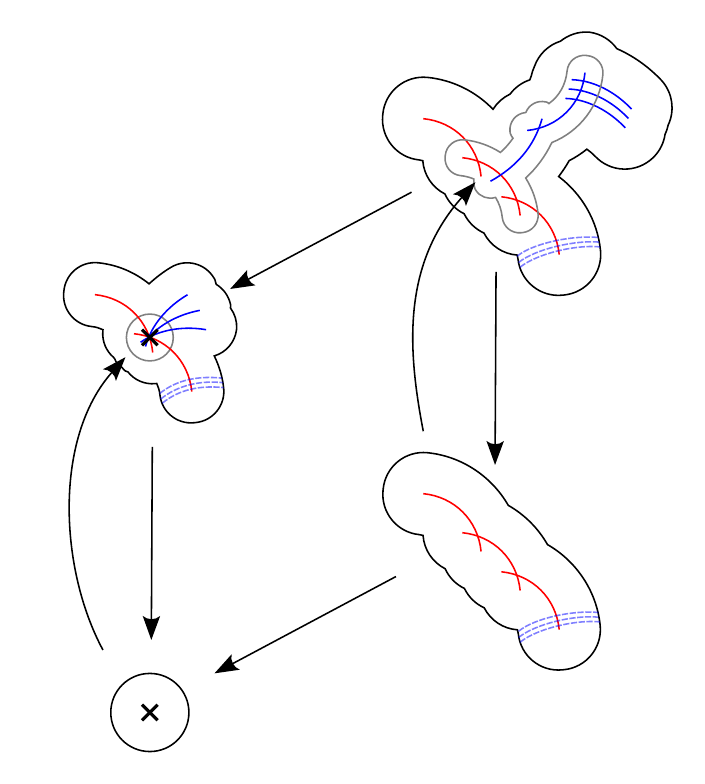
\hfill
\raisebox{3cm}{
\def\svgwidth{.32\columnwidth}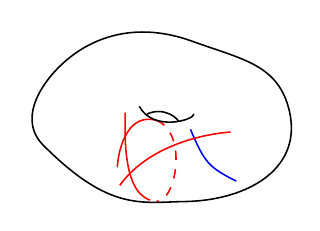
}
\caption{Kato datum over the quotient singularity of type $\frac{1}{5k-3}(1,2k-1)$.\\
The dashed curves are the critical curves and their strict transforms.
On the right side, the associated Kato surface $S$.
Self-intersections on curves are indicated for $Z$, $Z'$ and $S$, they are omitted for ($-1$)-curves.
}\label{fig:quot}
\end{figure}
\end{center}

By construction, we have that $f_* \ord_{E_1} = \ord_{E_2}$.
We deduce that
\[
\Kld(f_* \ord_{E_1}) - \Kld(\ord_{E_1}) = \Kld(\ord_{E_2})-\Kld(\ord_{E_1}) = -\frac{k-3}{p}\text{.}   
\]
In particular, this difference is positive when $k=2$, equals $0$ when $k=3$, and is negative when $k \geq 4$.

The Jacobian divisor $\JD{f}$ is supported on the contracter curves, whose strict transforms are the preimages by $\hat{\sigma}$ of the germs of $E_1$, $E_3$ and $D'_{j}$ at $E'$ for $j=1, \ldots, k-1$.
We denote such curves as $C_1, C_3$ and $D_j$ for $j=1, \ldots, k-1$ respectively.
Set $D=D_1+\cdots + D_{k-1}$. From the discussion above we deduce that
\[
\JD{f} = b_1 C_1 + b_3 C_3 + b D
\]
for suitable $b_1,b_3, b \in \nQ$.
We can compute these coefficients by evaluating the Jacobian formula \refeqn{jacobian_formula} at valuations whose normalization lies in the segments $[\nu_{E_1},\nu_{C_1}]$, $[\nu_{E_1}, \nu_{C_3}]$ and $[\nu_{E_3}, \nu_{D_j}]$ respectively.
For example, we have
\begin{align*}
b_1 =& \Kld(f_* \ord_{E_{1,1}}) - \Kld(f_*\ord_{E_1}) - \big(\Kld(\ord_{E_{1,1}}) - \Kld(\ord_{E_1})\big) \\
=& \big(\Kld(E_1)+\Kld(E_2)\big) - \Kld(E_2) - \big(\Kld(E_1) +1 - \Kld(E_1)\big)\\
=& \Kld(E_1)-1\text{,}
\end{align*}
and similarly
\[
b_3= \Kld(E_3) -1 \text{,}\qquad
b = \Kld(E_2) +2\text{.}
\]

\subsection{Remarks on Kato surfaces and eigenvaluations}

Let $f \colon (X,x_0) \selfarrow$ be a Kato germ that is also \emph{contracting}. In this case we can take representatives of the germs of the modificatoin $\Kpi \colon X' \to (X,x_0)$ and the local isomorphism $\Ksigma \colon (X,x_0) \to (X',x_1)$
so that $\Ksigma \colon X \to \Ksigma(X) \subrelcpct X'=\Kpi^{-1}(X)$ is an isomorphism.
Then we can construct the compact smooth surface
\[
S=S(\Kpi,\Ksigma):=\nfrac{X' \setminus \Ksigma(X)}{\Ksigma \circ \Kpi \colon \partial X' \to \Ksigma(\partial X)}\text{,}
\]
called the \emph{Kato surface} associated to the Kato datum $(\Kpi,\Ksigma)$.
Notice that, given the Kato germ $f$, the isomorphism class of $S$ depends on the Kato datum chosen, but not the bimeromorphic class of $S$, that we denote $S(f)$.
In the regular case $(X,x_0)\cong(\nC^2,0)$, the Kato datum that produces a minimal surface is essentially unique: it is one for which there is only one ($-1$)-curve in $X'$, which contains $x_1$.

Kato surfaces are compact, have negative Kodaira dimension and first Betti number $b_1=1$ (in particular, they are non-K\"ahler), which places them in class \textup{VII} of Kodaira's classification.
Their minimal model has positive second Betti number $b_2>0$ and conjecturally they are the only ones with these properties (see \cite{teleman:towardsclassification} and references therein).

\emph{Regular} Kato surfaces (i.e., associated to a Kato germ defined over $(\nC^2,0)$) have been introduced by Masahide Kato \cite{kato:cptcplxmanifoldsGSS}; their geometry is quite well understood, thanks to the work of Kato, Nakamura, Dloussky, Oeljeklaus, Toma, Teleman, etc.
In the singular setting, \cite{kato:cpctcplxsurfwithGSPH} classifies surfaces admitting a ``global strongly pseudoconvex hypersurfaces", that in our case appear as an embedding of the link of the singularity $(X,x_0)$ inside $S$ (see \cite[Theorems B and C]{fantini-favre-ruggiero:sandwichselfsim}).

We have seen (see \refrem{semiconj}) that any Kato germ is semiconjugated via a dominant bimeromorphic map to a regular Kato germ.
Regular Kato germs form a special class of contracting \emph{rigid} germs (for which the \emph{generalized critical set} $\displaystyle \mc{C}(f^\infty):= \bigcup_{n \in \nN^*}\mc{C}(f^n)$ has simple normal crossings and is $f$-invariant): they are exactly the ones that are strict by \refthm{strictkato}.

\begin{rem}
In higher dimensions, this is not true anymore: there are examples of regular Kato germs with infinitely many irreducible components in $\mc{C}(f^\infty)$.
\end{rem}

\begin{ex}\label{ex:nonrigidkatoSNC}
Let us consider the composition of the following two blow-ups:
\begin{itemize}
\item $\pi_1 \colon X_1 \to (\nC^3,0)$ the blow-up the origin; pick the $x$-chart for which $\pi_1(x_1,y_1,z_1)=(x_1,x_1y_1,x_1z_1)$.
\item $\pi_2 \colon X_2 \to X_1$ the blow-up of the curve $C=\{x=y-z^2=0\}$; in suitable coordinates centered at a point $p$, we write $\pi_2(x_2,y_2,z_2)=(x_2, x_2y_2+z_2^2, z_2)$.
\end{itemize}
Set $\Kpi=\pi_1 \circ \pi_2$.
With respect to the coordinates chosen above, consider $\Ksigma\colon (\nC^3,0) \to (X_2,p)$ given by $(x,y,z) \mapsto (y,x,z)$.
Then $(\Kpi,\Ksigma)$ defines a Kato datum, with associated Kato germ
\[
f(x,y,z)=\big(y,y(xy+z^2),yz\big).
\]
By direct computation we get $\mc{C}(f)=\{y=0\}$, and $\mc{C}(f^n)=\big\{y \prod_{j=1}^{n-1} (xy+jz^2)=0\big\}$. In particular, $f$ is not rigid.
\end{ex}

Contracting rigid germs are classified up to holomorphic conjugacy by Favre's \cite{favre:rigidgerms} (see also \cite{ruggiero:rigidgerms} for a partial result in higher dimensions). For contracting Kato germs, we get the following statement.

\begin{thm}[{\cite[Proposition 1.5 and Table II]{favre:rigidgerms}}]\label{thm:strictFclass}
Any strict germ $f \colon (X,x_0) \selfarrow$ over a normal surface singularity is semiconjugated via a dominant bimeromorphic map to a germ $\wt{f} \colon (\nC^2,0) \selfarrow$ of the following form.
\begin{enumerate}
\item[Class 2:]\[\wt{f}(z,w)=\big(\lambda z, z^c w + P(z)\big)\text{,}\]\vspace{2mm}
where $0 < \abs{\lambda} < 1$, $c \geq 1$, and $P \in z\nC[z]$ with $\deg P \leq c$;  
\item[Class 4:] \[\wt{f}(z,w)=\big(z^a, \mu z^c w + \overbrace{\sum_{k=1}^c a_k x^k}^{P(z)} + \eps z^{\frac{ac}{a-1}}\big)\text{,}\]\vspace{2mm}
where $a \geq 2$, $c \geq 1$, $P\in z \nC[z]$ such that $\gcd\{a, k \leq c\ : \ a_k \neq 0\} = 1$, and either:
\begin{trivlist}
\item[non-special case:] $\eps = 0$, $\mu \in \nC^*$; or
\item[special case:] $\mu=1$, $\frac{ac}{a-1} \in \nN^*$, $\eps \in \nC$.
\end{trivlist}
\item[Class 6:] \[\wt{f}(z,w)=\big(z^a w^b, z^c w^d\big)\text{,}\]\vspace{2mm}
where $A=\begin{pmatrix}a & b \\ c & d\end{pmatrix}$ satisfies $\det A = \pm 1$ and $A^2$ has positive entries. 
\end{enumerate}
\end{thm}
These three classes can also be characterized by the type of their eigenvaluations: $\wt{f}$ belongs to Class $2$, $4$ or $6$ if and only if its eigenvaluation $\wt{\nu}_\star=\nu_\star(\wt{f})$ is a curve semivaluation, infinitely singular, or irrational respectively.
Since the type of eigenvaluation (being non-divisorial) is preserved by semiconjugacy, the same characterization holds if we look at the type of eigenvaluation $\nu_\star =\nu_\star(f)$ for the original germ $f$.

Kato surfaces are divided into several classes, depending on their geometry (see \cite{dloussky:phdthesis}):
\begin{itemize}
\item \emph{Enoki surfaces}: there is a cycle $D$ of rational curves of self-intersection $D^2=0$ (each irreducible component has a self-intersection $-2$).
This class can be further divided into two subclasses: \emph{parabolic Inoue}, which also admit a non-rational compact curve $E$. This curve is disjoint from $C$, has genus $1$, and self-intersection $E^2=-b_2$, where $b_2$ is the number of irreducible components in the cycle $D$. Surfaces that have not such a compact curve $E$ are called \emph{special Enoki}.

\item \emph{Inoue-Hirzebruch surfaces}: they are characterized by the fact that the set of compact curves forms either one or two disjoint cycles of rational curves with negative self-intersection. The first case is called \emph{half Inoue}, while the second is called \emph{hyperbolic Inoue}.
\item \emph{intermediate Kato surfaces}: the remaining case, for which the compact curves are organized as a cycle of rational curves, with trees of rational curves attached to it.  
\end{itemize}

By looking at the decomposition of the normal forms $\wt{f} \colon (\nC^2,0) \selfarrow$ into Kato data $\wt{f}=\piY \circ \sigmaY$, and consequently to the geometry of the compact curves in the associated Kato surface $S$, one get the following characterization in terms of the type of eigenvaluation $\nu_\star$ of $\wt{f}$, or equivalently of $f$.

\begin{prop}
Let $f\colon (X,x_0) \selfarrow$, and let $S$ be the associated Kato surface.
Let $\nu_\star$ be the eigenvaluation for $f$. Then:
\begin{enumerate}
\item $S$ is a Enoki surface if and only if $\nu_\star=\nu_C$ is a curve semivaluation.
It is parabolic Inoue if and only if the curve $C$ is analytic.
\item $S$ is a Inoue-Hirzebruch surface if and only if $\nu_\star$ is irrational.
It is hyperbolic Inoue (resp., half Inoue) if and only if $f_\bullet$ preserves the connected components of $\mc{V} \setminus \{\nu_\star\}$ (resp., switches the connected components).
\item $S$ is a intermediate Kato surface if and only if $\nu_\star$ is infinitely singular.
\end{enumerate}
\end{prop}

\begin{rem}
Notice that the classes of Kato surfaces described above are invariant by replacing $f$ with an iterate $f^n$ (one can see it geometrically, via a natural Galois covering $S(f^n) \to S(f)$). The only exception is given by even order coverings of half Inoue surfaces, that are hyperbolic Inoue.
\end{rem}

\begin{rem}
One could wonder if the existence of a strict germ with a given type of eigenvaluation restricts furthermore the singularity $(X,x_0)$. This is not the case: any sandwiched singularity can support strict germs with eigenvaluations of the three allowed types.
One way to construct examples has been described in \refssec{sandwichedselfsim}: by picking opportunely the point $y_1 \in Y'$ and the local isomorphism $\sigmaY$, one can construct examples for the three types of eigenvaluations.
The main reason lying inside this phenomenon, is that typically, even if the modification $\Kpi$ is minimal (in the sense that the only ($-1$)-curves in $X'$ contain the point $x_1$), the Kato surface $S(\Kpi,\Ksigma)$ might not be minimal when $(X,x_0)$ is singular.

\end{rem}

\begin{ex}
In the example of \refssec{strictnoeff}, we get a Kato surface $S=S(\Kpi,\Ksigma)$ with $k+1$ rational curves (see \reffig{quot}):
\begin{itemize}
\item a nodal curve $\ol{C}_1$, obtained from $E_1$ and the strict transform of $C_1$, of self-intersection $-1$;
\item $\ol{C}_1$ intersects transversely at a free point the curve $\ol{C}_3$, obtained from $E_3$ and the strict transform of $C_3$, of self-intersection $-k$;
\item finally $\ol{C}_3$ intersects $k-1$ curves $\ol{D}_j$, obtained from $D'_j$ and the strict transform of $D_j$, of self-intersection $-1$.
\end{itemize}
The minimal model for $S$, obtained by contracting $\ol{D}_j$ and $\ol{C}_3$, contains only the nodal curve with self-intersection $0$.
We are hence in the case of a Enoki surface, and the eigenvaluation $\nu_\star$ is a curve semivaluation.
\end{ex}


\small
\bibliographystyle{alpha}

\bibliography{biblio}

\def\cprime{$'$}
\begin{thebibliography}{AGZV85}

\bibitem[AGZV85]{arnold-guseinzade-varchenko:singdiffmaps1}
V.~I. Arnold, S.~M. Gusein-Zade, and A.~N. Varchenko.
\newblock {\em Singularities of differentiable maps. {V}ol. {I}}, volume~82 of
  {\em Monographs in Mathematics}.
\newblock Birkh\"auser Boston, Inc., Boston, MA, 1985.
\newblock The classification of critical points, caustics and wave fronts,
  Translated from the Russian by Ian Porteous and Mark Reynolds.

\bibitem[BdFF12]{boucksom-defernex-favre:volumeisolatedsing}
Sebastien Boucksom, Tommaso de~Fernex, and Charles Favre.
\newblock The volume of an isolated singularity.
\newblock {\em Duke Math. J.}, 161(8):1455--1520, 2012.

\bibitem[Ber90]{berkovich:book}
Vladimir~G. Berkovich.
\newblock {\em Spectral theory and analytic geometry over non-{A}rchimedean
  fields}, volume~33 of {\em Mathematical Surveys and Monographs}.
\newblock American Mathematical Society, Providence, RI, 1990.

\bibitem[BH14]{broustet-horing:singendo}
Ama\"{e}l Broustet and Andreas H\"{o}ring.
\newblock Singularities of varieties admitting an endomorphism.
\newblock {\em Math. Ann.}, 360(1-2):439--456, 2014.

\bibitem[CMS09]{camacho-movasati-scardua:quasihomosteinsurfsing}
C.~Camacho, H.~Movasati, and B.~Sc{\'a}rdua.
\newblock The moduli of quasi-homogeneous {S}tein surface singularities.
\newblock {\em J. Geom. Anal.}, 19(2):244--260, 2009.

\bibitem[dF19]{defelipe:topspacesval}
Ana~Bel\'{e}n de~Felipe.
\newblock Topology of spaces of valuations and geometry of singularities.
\newblock {\em Trans. Amer. Math. Soc.}, 371(5):3593--3626, 2019.

\bibitem[DFR25]{dujardin-favre-ruggiero:polyskewprod}
Romain Dujardin, Charles Favre, and Matteo Ruggiero.
\newblock Polynomial skew products with small relative degree, 2025.

\bibitem[Dlo84]{dloussky:phdthesis}
Georges Dloussky.
\newblock Structure des surfaces de {K}ato.
\newblock {\em M\'em. Soc. Math. France (N.S.)}, (14):ii+120, 1984.

\bibitem[Fan14]{fantini:normalizedlinks}
Lorenzo Fantini.
\newblock Normalized non-{A}rchimedean links and surface singularities.
\newblock {\em C. R. Math. Acad. Sci. Paris}, 352(9):719--723, 2014.

\bibitem[Fan18]{fantini:normalizedBerkovich}
Lorenzo Fantini.
\newblock Normalized {B}erkovich spaces and surface singularities.
\newblock {\em Trans. Amer. Math. Soc.}, 370(11):7815--7859, 2018.

\bibitem[Fav00]{favre:rigidgerms}
Charles Favre.
\newblock Classification of 2-dimensional contracting rigid germs and {K}ato
  surfaces. {I}.
\newblock {\em J. Math. Pures Appl. (9)}, 79(5):475--514, 2000.

\bibitem[Fav10]{favre:holoselfmapssingratsurf}
Charles Favre.
\newblock Holomorphic self-maps of singular rational surfaces.
\newblock {\em Publ. Mat.}, 54(2):389--432, 2010.

\bibitem[FFR20]{fantini-favre-ruggiero:sandwichselfsim}
Lorenzo Fantini, Charles Favre, and Matteo Ruggiero.
\newblock Links of sandwiched surface singularities and self-similarity.
\newblock {\em Manuscripta Math.}, 162(1-2):23--65, 2020.

\bibitem[Fis76]{fischer:cplxanalgeom}
Gerd Fischer.
\newblock {\em Complex analytic geometry}, volume Vol. 538 of {\em Lecture
  Notes in Mathematics}.
\newblock Springer-Verlag, Berlin-New York, 1976.

\bibitem[FJ04]{favre-jonsson:valtree}
Charles Favre and Mattias Jonsson.
\newblock {\em The valuative tree}, volume 1853 of {\em Lecture Notes in
  Mathematics}.
\newblock Springer-Verlag, Berlin, 2004.

\bibitem[FJ07]{favre-jonsson:eigenval}
Charles Favre and Mattias Jonsson.
\newblock Eigenvaluations.
\newblock {\em Ann. Sci. \'Ecole Norm. Sup. (4)}, 40(2):309--349, 2007.

\bibitem[FR14]{favre-ruggiero:normsurfsingcontrauto}
Charles Favre and Matteo Ruggiero.
\newblock Normal surface singularities admitting contracting automorphisms.
\newblock {\em Ann. Fac. Sci. Toulouse Math. (6)}, 23(4):797--828, 2014.

\bibitem[GPR94]{grauert-peternell-remmert:sevcplxvar7}
H.~Grauert, Th. Peternell, and R.~Remmert, editors.
\newblock {\em Several complex variables. {VII}}, volume~74 of {\em
  Encyclopaedia of Mathematical Sciences}.
\newblock Springer-Verlag, Berlin, 1994.
\newblock Sheaf-theoretical methods in complex analysis, A reprint of {\it
  Current problems in mathematics. Fundamental directions. Vol.\ 74} (Russian),
  Vseross.\ Inst.\ Nauchn.\ i Tekhn.\ Inform.\ (VINITI), Moscow.

\bibitem[GR21]{gignac-ruggiero:locdynnoninvnormsurfsing}
William Gignac and Matteo Ruggiero.
\newblock Local dynamics of non-invertible maps near normal surface
  singularities.
\newblock {\em Mem. Amer. Math. Soc.}, 272(1337), 2021.

\bibitem[Gra62]{grauert:ubermodifikationen}
Hans Grauert.
\newblock \"{U}ber {M}odifikationen und exzeptionelle analytische {M}engen.
\newblock {\em Math. Ann.}, 146:331--368, 1962.

\bibitem[Hir73]{hironaka:introrealanal}
Heisuke Hironaka.
\newblock {\em Introduction to real-analytic sets and real-analytic maps}.
\newblock Quaderni dei Gruppi di Ricerca Matematica del Consiglio Nazionale
  delle Ricerche. Istituto Matematico ``L. Tonelli'' dell'Universit\`a di Pisa,
  Pisa, 1973.

\bibitem[Hir75]{hironaka:flattening}
Heisuke Hironaka.
\newblock Flattening theorem in complex-analytic geometry.
\newblock {\em Amer. J. Math.}, 97:503--547, 1975.

\bibitem[HLJT73]{hironaka-lejeunejalabert-teissier:platificateurlocal}
Heisuke Hironaka, Monique Lejeune-Jalabert, and Bernard Teissier.
\newblock Platificateur local en g\'{e}om\'{e}trie analytique et aplatissement
  local.
\newblock In {\em Singularit\'{e}s \`a {C}arg\`ese ({R}encontre
  {S}ingularit\'{e}s {G}\'{e}om. {A}nal., {I}nst. \'{E}tudes {S}ci. de
  {C}arg\`ese, 1972)}, pages 441--463. 1973.

\bibitem[IOPR22]{istrati-otiman-pontecorvo-ruggiero:torickatomanifolds}
Nicolina Istrati, Alexandra Otiman, Massimiliano Pontecorvo, and Matteo
  Ruggiero.
\newblock Toric {K}ato manifolds.
\newblock {\em J. \'{E}c. polytech. Math.}, 9:1347--1395, 2022.

\bibitem[Jon12]{jonsson:berkovich}
Mattias Jonsson.
\newblock Dynamics on {B}erkovich spaces in low dimensions.
\newblock 2012.
\newblock arXiv:1201.1944v1.

\bibitem[Kat78]{kato:cptcplxmanifoldsGSS}
Masahide Kato.
\newblock Compact complex manifolds containing ``global'' spherical shells.
  {I}.
\newblock In {\em Proceedings of the {I}nternational {S}ymposium on {A}lgebraic
  {G}eometry ({K}yoto {U}niv., {K}yoto, 1977)}, pages 45--84, Tokyo, 1978.
  Kinokuniya Book Store.

\bibitem[Kat79]{kato:cpctcplxsurfwithGSPH}
Masahide Kato.
\newblock Compact complex surfaces containing global strongly pseudoconvex
  hypersurfaces.
\newblock {\em T\^ohoku Math. J. (2)}, 31(4):537--547, 1979.

\bibitem[KM98]{kollar-mori:biratgeomalgvar}
J{\'a}nos Koll{\'a}r and Shigefumi Mori.
\newblock {\em Birational geometry of algebraic varieties}, volume 134 of {\em
  Cambridge Tracts in Mathematics}.
\newblock Cambridge University Press, Cambridge, 1998.
\newblock With the collaboration of C. H. Clemens and A. Corti, Translated from
  the 1998 Japanese original.

\bibitem[Kol13]{kollar:singMMP}
J\'anos Koll\'ar.
\newblock {\em Singularities of the minimal model program}, volume 200 of {\em
  Cambridge Tracts in Mathematics}.
\newblock Cambridge University Press, Cambridge, 2013.
\newblock With a collaboration of S\'andor Kov\'acs.

\bibitem[Mat02]{matsuki:mori-program}
Kenji Matsuki.
\newblock {\em Introduction to the {M}ori program}.
\newblock Universitext. Springer-Verlag, New York, 2002.

\bibitem[Mor24]{morvan:singadmitcontrauto}
Kémo Morvan.
\newblock Singularities admitting contracting automorphisms, 2024.

\bibitem[M{\"u}l87]{muller:liegroupsanalyticalgebras}
Gerd M{\"u}ller.
\newblock Actions of complex {L}ie groups on analytic
  {${\textbf{C}}$}-algebras.
\newblock {\em Monatsh. Math.}, 103(3):221--231, 1987.

\bibitem[OV24]{ornea-verbitsky:algconesLCKmflspot}
Liviu Ornea and Misha Verbitsky.
\newblock Algebraic cones of lck manifolds with potential.
\newblock {\em Journal of Geometry and Physics}, 198:105103, April 2024.

\bibitem[OW71]{orlik-wagreich:isolatedsingagsurfCaction}
Peter Orlik and Philip Wagreich.
\newblock Isolated singularities of algebraic surfaces with {C{$^{\ast}$}}\
  action.
\newblock {\em Ann. of Math. (2)}, 93:205--228, 1971.

\bibitem[Rug13]{ruggiero:rigidgerms}
Matteo Ruggiero.
\newblock Contracting rigid germs in higher dimensions.
\newblock {\em Ann. Inst. Fourier (Grenoble)}, 63(5):1913--1950, 2013.

\bibitem[Rug25]{ruggiero:HDR}
Matteo Ruggiero.
\newblock {\em Dynamical singularities}.
\newblock {HDR} memoir, Universit\'e Paris Cité - IMJ-PRG - ED 386, 2025.
\newblock Habilitation à Diriger des Recherches.

\bibitem[Spi90]{spivakovsky:sandwichedsing}
Mark Spivakovsky.
\newblock Sandwiched singularities and desingularization of surfaces by
  normalized {N}ash transformations.
\newblock {\em Ann. of Math. (2)}, 131(3):411--491, 1990.

\bibitem[Tel17]{teleman:towardsclassification}
Andrei Teleman.
\newblock Towards the classification of class {VII} surfaces.
\newblock In {\em Complex and symplectic geometry}, volume~21 of {\em Springer
  INdAM Ser.}, pages 249--262. Springer, Cham, 2017.

\bibitem[Thu07]{thuillier:homotopy}
Amaury Thuillier.
\newblock G\'eom\'etrie toro\"\i dale et g\'eom\'etrie analytique non
  archim\'edienne. {A}pplication au type d'homotopie de certains sch\'emas
  formels.
\newblock {\em Manuscripta Math.}, 123(4):381--451, 2007.

\bibitem[Wah90]{wahl:charnumlinksurfsing}
Jonathan Wahl.
\newblock A characteristic number for links of surface singularities.
\newblock {\em J. Amer. Math. Soc.}, 3(3):625--637, 1990.

\bibitem[Zha17]{zhang:isolatedsingnoninvfiniteendo}
Yuchen Zhang.
\newblock On isolated singularities with a noninvertible finite endomorphism.
\newblock {\em Adv. Math.}, 308:859--878, 2017.

\end{thebibliography}

\end{document}